\documentclass{amsart}
\usepackage{graphicx} 
\usepackage{amssymb}
\usepackage{xcolor}
\usepackage{amsthm}
\usepackage{float}
\usepackage{amsmath,amsfonts,amssymb}
\usepackage{tikz-cd}
\usepackage[shortlabels]{enumitem}
\setlist[enumerate]{leftmargin=*}
\usepackage{mathrsfs}   

\usepackage{fullpage}

\newcommand{\wh}[1]{\widehat{#1}}

\newcommand{\wt}[1]{\widetilde{#1}}

\newcommand{\Map}{\operatorname{Map}}
\newcommand{\wei}{\operatorname{wt}}

\newcommand{\Spec}{\operatorname{Spec}}
\newcommand{\Pic}{\operatorname{Pic}}
\newcommand{\Ann}{\operatorname{Ann}}
\newcommand{\Res}{\operatorname{Res}}

\newcommand{\lcm}{\operatorname{lcm}}
\newcommand{\id}{\operatorname{id}}
\newcommand{\val}{\operatorname{val}}
\newcommand{\Ter}{\operatorname{Ter}}
\newcommand{\opp}{\operatorname{op}}
\newcommand{\Opp}{^{\opp}\mc O}
\renewcommand{\div}{\operatorname{div}}
\newcommand{\Proj}{\operatorname{Proj}}

\newcommand{\mc}{\mathcal}
\newcommand{\ov}{\overline}
\newcommand{\un}{\underline}
\newcommand{\bb}{\mathbb}

\newtheorem{theorem}{Theorem}[section]
\newtheorem{definition}[theorem]{Definition}
\newtheorem{lemma}[theorem]{Lemma}
\newtheorem{proposition}[theorem]{Proposition}
\newtheorem{corollary}[theorem]{Corollary}
\newtheorem{remark}[theorem]{Remark}
\newtheorem{question}[theorem]{Question}

\newenvironment{sis}{\left\{\begin{aligned}}{\end{aligned}\right.}

\usepackage[backref=page]{hyperref}

\hypersetup{
 colorlinks,
 citecolor=green,
 linkcolor=blue,
 urlcolor=blue}

\usepackage{booktabs}
\usepackage{color}

\numberwithin{equation}{section}

\usepackage{tikz}
\usetikzlibrary{decorations.pathmorphing}
\tikzset{dot/.style={
		circle,
		fill=black,
		inner sep=1.5pt,
}}

\title{The Hassett-Keel program via $\Theta$-stability}

\author{Luca Tasin}
\address{Dipartimento di Matematica F.\ Enriques, Universit\`a degli Studi di Milano, Via Cesare Saldini 50, 20133 Milano, Italy} 
\email{luca.tasin@unimi.it}

\author{Filippo Viviani}
\address{Dipartimento di Matematica, Universit\`a degli Studi di Roma Tor Vergata, Via della Ricerca Scientifica, 00133 Roma, Italy.}
\email{viviani@mat.uniroma2.it}

\date{\today}

\begin{document}

\begin{abstract}
We study the intrinsic notion of $\alpha$-stability for curves arising from the Beyond GIT approach to moduli spaces. We show that it recovers Deligne--Mumford stability for $9/11<\alpha\le1$, and more generally that the $\alpha$-semistable locus is contained in the known modular compactification of the Hassett-Keel proram for $\alpha>2/3-\varepsilon$. As applications, we also obtain new $\alpha$-stability results for smooth curves. Our approach combines slope inequalities with a degeneration theorem showing that every Gorenstein curve with a non-nodal singularity can be isotrivially degenerated to a Gorenstein curve with a $\mathbb G_m$-action.
\end{abstract}

\maketitle

\setcounter{tocdepth}{1}
\tableofcontents

\section{Introduction}

Let  $\ov{\mc M}_{g,n}$  be the stack of stable $n$-pointed curves of genus $g$ and denote with $\ov M_{g,n}$ its coarse moduli space. For any rational number $\alpha \in [0,1]$ consider the projective variety 
$$
\ov M_{g,n}(\alpha):= \mathrm{Proj}\bigoplus_{m \ge 0} H^0(\ov M_{g,n}, m(K_{\ov{\mc M}_{g,n}} + \alpha\delta+(1-\alpha)\psi)),
$$ 
where $\delta$ is the total boundary divisor and $\psi$ is the total cotangent class.  
The Hassett-Keel program aims to provide a modular interpretation of $\ov M_{g,n}(\alpha)$, namely to realize them as good moduli spaces of appropriate stacks of curves, and to use their modular interpretation in order to study the rational map $f_{\alpha}: \ov M_{g,n}\dashrightarrow \ov M_{g,n}(\alpha)$. 
The first two steps of the Hassett-Keel program (which cover the case where $\alpha \geq 7/10-\epsilon$ with $0<\epsilon \ll 1$) were constructed by Hassett-Hyeon in \cite{HH1, HH2} for $n=0$, using GIT of the Hilbert/Chow scheme of pluticanonical embedded curves. 
To construct the third step (and to generalize the previous cases to $n>0$), in the trilogy  \cite{AFSV1,AFS2,AFS3}, Alper-Fedorchuk-Smyth-van der Wyck introduced a new technology using local variation of GIT for stacks. For $\alpha \ge 2/3 - \varepsilon$ with $0 < \varepsilon \ll 1$, they defined algebraic stacks $\ov{\mc M}_{g,n}(\alpha)$ parametrizing suitable singular curves whose good moduli spaces are exactly $\ov M_{g,n}(\alpha)$, see Definition \ref{D:Mgalpha}.
Based on the same methods, the first steps of the MMP of $\ov M_{g}$ have been studied in \cite{CTV21,CTV23a, Gori}. 
For $g \le 4$, the Hassett-Keel program of $\ov M_g$ has been fully understood, see \cite{Has0, HL10, CMJL12, Fedorchuk12, CMJL14, HL14, LZ25, ADLW}. See also  \cite{zhao2023} for partial recent progress in the genus 6 case. 

Recently, a more intrinsic approach to the study of moduli problems has been developed, based on the so-called ``Beyond GIT'' framework introduced by Halpern-Leistner; see \cite{HL} for a comprehensive treatment of the theory and \cite{GMP1, GMP2} for an application to the theory of $A_r$-stable curves. For a recent survey on this approach, including several interesting open questions, see \cite{AHL}.

Within this framework, we adopt the following set-up. Let $\mathcal U_{g,n}$ be the stack of Gorenstein reduced connected $n$-pointed (i.e. with $n$ marked pairwise distinct and smooth points $\{p_i\}$) projective curves $(C,p_i)$ of genus $g$ with $\omega_C^{log}:=\omega_C(\sum p_i)$ ample. The integral stack $\ov {\mc M}_{g,n}$ of stable $n$-pointed curves of genus $g$ is an open substack of $\mc U_{g,n}$, and the closure of $\ov{\mc M}_{g,n}$ in $\mc U_{g,n}$ is the irreducible component parametrizing smoothable curves $(C,p_i)$ of $\mc U_{g,n}$. The canonical divisor $K$ of $\ov{\mc M}_{g,n}$, the boundary divisor $\delta$ and the total cotangent class $\psi$ can be extended from $\ov {\mc M}_{g,n}$ to $\mc U_{g,n}$ (see \S \ref{Sub:stack-curves} for more details).
Given $\alpha \in [0,1]$, we say that  curve $C \in  \mathcal U_{g,n}$ is \emph{$\alpha$-semistable} if for any test configuration $T$ in $\mc U_{g,n}$ with general fibre $C$, the weight $\wei_T(K+\psi+\alpha(\delta-\psi))$ is non-negative. We denote by $\mathcal U_{g,n}(\alpha)$ the locus in $\mathcal U_{g,n}$ consisting of smoothable $\alpha$-semistable curves (see \S\ref{Sub:Theta-stability} for more details).

The following natural question could provide an interpretation of the Hassett-Keel program via $\Theta$-stability.

\begin{question}\label{q:alpha-stability}
Given $\alpha\in [0,1]$, is there a good moduli space morphism $\phi_{\alpha}:\mathcal U_{g,n}(\alpha) \to \ov M_{g,n}(\alpha)$ in such a way that $\phi_{\alpha}^*\mc O_{\ov M_{g,n}(\alpha)}(1)=K+\psi+\alpha(\delta-\psi)$?  
\end{question}
The above Question provides a concrete proposal for the Modularity Principle of \cite[Principle 1.2]{AFS16} and it was raised in \cite[Sec. 6.3]{AHL}.
Note that, a priori, it is not even clear that $\mc U_{g,n}(\alpha)$ is locally closed in $\mc U_{g,n}$ and hence that it has the structure of an algebraic stack: the above Question requires this property to hold. Moreover, there could be other choices for the universe stack $\mc U_{g,n}$ in which one could ask the same question (e.g. allowing non-reduced or non-Gorenstein curves): the choice of our universe stack is dictated by the techniques used in the proof of our results and also by the set-up in \cite{AFS16}.

Our main result is the following, which gives a positive answer to Question \ref{q:alpha-stability} for large values of $\alpha$, thus solving positively one of the open questions in \cite[Sec. 6.3]{AHL}.

\begin{theorem}[Theorem \ref{T:stab} + Corollary \ref{C:destab9/11}]\label{thm:main}
A smoothable curve  $C \in \mathcal U_{g,n}$ is $\alpha$-semistable for $9/11 < \alpha \le 1$ if and only if it is stable. In other words, we have that  
$$\mathcal U_{g,n}(\alpha)=\ov{\mc M}_{g,n} \quad \text{ for } \alpha \in (9/11,1].
$$
\end{theorem}

The above Theorem is expected to be sharp, in view of Question \ref{q:alpha-stability},  since the rational map $\ov M_{g,n}\dashrightarrow \ov M_{g,n}(\alpha)$ is an isomorphism if and only if $9/11 < \alpha \le 1$ (see \cite[Chap. XIV, Thm. 5.2]{GAC2}).
We now explain in detail the proof of the two implications in Theorem \ref{thm:main}, and the several byproducts that we get. 

In order to prove the  if implication in Theorem \ref{thm:main}, i.e. the statement that a smoothable non-nodal curve $C \in \mathcal U_{g,n}$ is not $\alpha$-semistable for $9/11< \alpha \le 1$,
the crucial ingredient is the following Theorem which says that we can always isotrivially degenerate a Gorenstein curve with a non-nodal singularity to a Gorenstein curve with $\mathbb G_m$-action which is concentrated on an atom of a Gorenstein curve singularity with $\bb G_m$-action.  

\begin{theorem}\label{thm:deg-sing}(Corollary \ref{C:lim-O} + Theorem \ref{T:deg-sing})
 Let $(C,p_i)\in \mc U_{g,n}$ having a singular point $q\in C$ which is not a node. Then there exists a  Gorenstein (non-nodal) curve singularity $\mc O$ with $\bb G_m$-action with the same normalization and the same conductor ideal as $\wh{\mc O}_{C,q}$ and  a test configuration $(\mc C,\sigma_i)\to \bb A^1$ for $(C,p_i)$ such that 
 \begin{enumerate}[label=(\arabic*)]
     \item \label{thm:deg-sing1} $(\mc C_0,\sigma_i(0))$ is obtained from the pointed normalization $\nu: (C,p_i)^{\nu}\to (C,p)$ at $p$ and the projective atom $\ov{X}(\mc O)$ for $\mc O$ by gluing nodally the marked points $\{q_i\}$ of $\ov{X}(\mc O)$ with the points of $\nu^{-1}(q)$, and then contracting the components where the log-canonical line bundle is not ample.
\item \label{thm:deg-sing2} The action of $\bb G_m$ on $\mc C_0$, induced by the test configuration, is trivial on the image of $(C,p_i)^{\nu}$ and it is the standard action on the complementary subcurve, which is an $S$-dangling projective atom  $\ov{X}(\mc O)^S$.
 \end{enumerate}
\end{theorem}

Note that $\mc O$ and $\wh{\mc O}_{C,q}$ have the same numerical invariants, e.g.  the same delta-invariant, the same number of branches and the same genus. Moreover, the $S$-dangling projective atom $\ov X(\mc O)^S$ appearing in the central fiber $\mc C_0$ is obtained by gluing $b$ rational smooth curves (where $b$ is the number of branches of $\mc O$) at the origin in order to produce the singularity $\mc O$ and then attaching the irreducible components corresponding to $S^c$ to the rest of the curve. 
The proof goes like this: we first construct in Corollary \ref{C:lim-O} the curve singularity $\mc O$ with $\bb G_m$-action by taking the limit of $\mc {\wh O}_{C,q}$ in its normalization under the action of a suitable one-parameter subgroup (and we use the theory of territories \cite{BGS} to show that such a limit exists and it is Gorenstein); then, we use a push-out construction to glue the above local isotrivial specialization of $\mc {\wh O}_{C,q}$ into $\mc O$ with a suitable weighted blow-up of $(C,p_i)^{\nu}\times \bb A^1\to \bb A^1$ at $\nu^{-1}(q)\times \{0\}$.

Theorem \ref{T:deg-sing} is the counterpart of a construction of Chen-Yu \cite[Thm. 2.6, Thm. 2.7]{CY} (which we review in Theorem \ref{T:deg-CY} in a more general form), in which one starts from a subcurve $Z$ of $(C,p_i)$ meeting the complementary subcurve $Z^c$ in nodes and a regular differential on $Z$, and one constructs a test configuration for $(C,p_i)$ whose central fiber is obtained by gluing $Z^c$ and a certain S-dangling projective atom $\ov X(\Opp)^S$, whose action is opposite to the one of $\ov X(\mc O)^S$.

Next, in Theorem \ref{T:destab}, we use the above Theorem \ref{thm:deg-sing} and some formulas from \cite{AFS16} and \cite{CY} for the weights of the $\bb G_m$-action on the space of pluricanonical forms on smoothable dangling atoms $\ov X(\mc O)^S$, 
in order to classify all the possible singularities of curves contained in $\mc U(\alpha)$ for any $\alpha \ge 5/9$. Note that Theorem \ref{T:destab} agrees with the predictions in \cite[Table 3]{AFS16}, which provides further evidence for Question \ref{q:alpha-stability}.
From the classification in Theorem \ref{T:destab}, it follows that if $\alpha>9/11$ then an $\alpha$-semistable curve can only have nodal singularity (see Corollary \ref{C:destab9/11}), which concludes the proof of the if implication in Theorem \ref{thm:main}.

More generally, the same techniques allows us to prove that, for $\alpha>2/3-\epsilon$, our stack $\mc U_{g,n}(\alpha)$ is contained in the stack  $\ov{\mc M}_{g,n}(\alpha)$ constructed by Alper-Fedorchuck-Smyth-van der Wyck \cite{AFSV1, AFS2, AFS3}.  

\begin{theorem}[= Theorem \ref{thm:U(alpha)}]
Let  $\alpha \in (2/3-\varepsilon,1]$, where $0 < \varepsilon \ll 1$. Then
$$
\mc U_{g,n}(\alpha) \subset \ov{\mc M}_{g,n}(\alpha).
$$
\end{theorem}
Indeed, we expect that equality holds in the above Theorem (which is the case for $\alpha\in (9/11,1]$ by Theorem \ref{thm:main}). Moreover, as a byproduct of our proof, we get a precise prediction for the next threeshold value (after $2/3$), see Reamrk \ref{R:threshold}.

\medskip

The if implication in Theorem \ref{thm:main}, i.e. that any stable $n$-pointed nodal curve $(C,p_i)$ is $\alpha$-semistable for $9/11<\alpha \le 1$, is the content of Theorem \ref{T:stab}. The proof goes as follows. Consider a compactified test configuration $\ov f: (\ov{\mc C},\sigma_i) \to \bb P^1$ for $(C,p_i)$. Then the weight  $\wei_T(K+\psi+\alpha(\delta-\psi))$ can be computed in the terms of the self-intersection of the relative log-canonical bundle and the degree of the Hodge bundle (see Lemma \ref{l:weight-test}). The non-negativity follows from an application of a (so-called) slope inequality from \cite{CTV23}. We remark that the slope inequality which is used in the proof requires the general fibre of the test configuration to be at worst nodal (and so it can not be applied directly for lower values of $\alpha$) and the total space to be demi-normal (and so it cannot be applied if the special fiber is not generically reduced, see Lemma \ref{l:singularities-test}).

As another application of the slope inequalities in this context, we show the following stability result for (unpointed) smooth curves.

\begin{theorem}(see the end of Section \ref{S:stability-via-slope} and of Subsection \ref{Sub:CY})  \label{T:Xiao} Let $g\geq 2$ (and $n=0$).
\begin{enumerate}
    \item \label{T:Xiao1} If $\frac{3g+8}{8g+4} \leq \alpha$ then all smooth curves of genus $g$ are $\alpha$-semistable.
    \item \label{T:Xiao2} If $\alpha<\frac{3g+8}{8g+4}$ then no hyperelliptic curve of genus $g$ is $\alpha$-semistable. 
\end{enumerate}
In particular, we have that
$$
\mc M_g\subseteq \ov{\mc U}_g(\alpha)\Longleftrightarrow  \frac{3g+8}{8g+4} \leq \alpha.
$$
\end{theorem} 

The proof of part \eqref{T:Xiao1} of the above Theorem  uses the classical slope inequality of Xiao \cite{Xiao} and Cornalba-Harris \cite{CH88} in the general form obtained in \cite[Theorem E]{CTV23}. The proof of part \eqref{T:Xiao2} of the above Theorem uses the fact that any hyperelliptic curve of genus $g$ can be isotrivially degenerated to the atom $\ov X(^{\opp}A_{2g})^{\{1\}}$ or the atom $\ov X(^{\opp}A_{2g+1})^{\{1,2\}}$ (see subsection \ref{Sub:global} for the notation).

\begin{remark}\label{R:smooth-n0}
The rational map $\ov M_g\dashrightarrow \ov M_g(\alpha)$ is an isomorphism over $M_g$ if  $\alpha> \frac{3g+8}{8g+4}$, it is regular on $M_g$ and contracts the hyperelliptic locus $H_g$ for $\alpha=\frac{3g+8}{8g+4}$, and it is not defined on $H_g$ for $\alpha< \frac{3g+8}{8g+4}$   (see \cite[Cor. 2]{CLV}). Therefore, the above Theorem \ref{T:Xiao} is coherent with Question \ref{q:alpha-stability}.
\end{remark}

Indeed, it would be interesting to extend the above Theorem \ref{T:Xiao} (and the results of \cite{CLV} mentioned in Remark \ref{R:smooth-n0}) to the case $n>0$.

\subsection*{Acknowledgment}
We are deeply grateful to G. Codogni and J. Alper for their insightful and stimulating discussions on the topic of this paper. We thank D. Chen  and D. Gori for some useful comments on an early draft of the manuscript. 

This project was initiated during the workshop "Developments in moduli problems"  at the American Institute of Mathematics in San Jose in 2023, organized by Jarod Alper, Daniel Halpern-Leistner, Yuchen Liu, and FV. We thank AIM for its hospitality, the organizers and the participants of the workshop for many stimulating discussions.
 
 We thank L. Battistella, F. Carocci, X. Wang, and  J. Zhao for useful conversations.

FV is funded by the MUR  ``Excellence Department Project'' MATH@TOV, awarded to the Department of Mathematics, University of Rome Tor Vergata, CUP E83C18000100006, by the  PRIN 2022 ``Moduli Spaces and Birational Geometry''  funded by MUR,  and he is a member of  the GNSAGA section of INdAM.

LT is a member of  the GNSAGA section of INdAM.

\subsection*{AI disclosure}
No AI tools were used to conceive, write, or review any part of this manuscript.

\section{Preliminaries}	
	
We work over an algebraically closed field $k$.  For an introduction to the theory of stacks, with a particular focus on the moduli space of curves, we refer the reader to \cite{Alper}.

\subsection{Test configurations}\label{Sub:test}
Let $X$ be a projective scheme over $k$ and $L$ a $\bb Q$-line bundle on $X$. 

\begin{definition} A test configuration $\mc X$ for $X$ consists of the following data:
\begin{enumerate}
	\item a flat proper morphism of schemes $f: \mc X \to \bb A^1$;
	\item a $\bb G_m$-action on $\mc X$ lifting the standard action on $\bb A^1$;
	\item an isomorphism $\mc X_1\cong  X$.
\end{enumerate}	

A test configuration $(\mc X, \mc L)$ for $(X,L)$ is a test configuration $\mc X$ for $X$ together with  a $\bb G_m$-linearised $\bb Q$-line bundle $\mc L$ on $\mc X$ and an isomorphism $(\mc X_1, \mc L_1) \cong (X, L)$. 
\end{definition}

Given a test configuration $f: \mc X\to \bb A^1$ we can consider its compactification $\ov f: \ov{\mc X} \to \bb P^1$ obtained gluing it with a trivial family at infinity. 

\begin{lemma} \label{l:singularities-test}
Let $ \mc X$ be a test configuration for $X$. 
\begin{enumerate}
    \item \label{l:singularities-test1} $\mc X$ is $S_2$ if and only if $X$ is $S_2$ and $\mc X_0$ is $S_1$.
    \item \label{l:singularities-test2} If $X$ is smooth (resp. nodal) in codimension one and $\mc X_0$ is generically reduced then $\mc X$ is smooth (resp. nodal) in codimension one. 
\end{enumerate}
In particular, if $\mc X_0$ is reduced and $X$ is normal (resp. demi-normal) then $\mc X$ is normal (resp. demi-normal).
\end{lemma}
\begin{proof}
	Part \eqref{l:singularities-test1} is proved in \cite[Prop. 2.6(ii)]{BHJ17}.

    Part \eqref{l:singularities-test2} in the smooth case is proved in \cite[Prop. 2.6(iii)]{BHJ17}. The nodal case is treated in a similar way: let $\eta$ be a codimension one point of $\mc X$. Then there are two possibilities: either $\eta$ is a codimension one point of $\mc X\setminus \mc X_0$, in which case we conclude using that $\mc X\setminus \mc X_0\cong X\times (\bb A^1\setminus\{0\})$ is nodal in codimension one; or $\eta$ is codimension zero (i.e. a generic) point of $\mc X_0$ in which case we deduce that $\mc O_{\mc X,\eta}$ is regular since $\mc X_0$ is a Cartier divisor of $\mc X$ and $\mc O_{\mc X_0,\eta}\cong k$ is regular by assumption. 
\end{proof}	
	
\subsection{Stacks of curves}\label{Sub:stack-curves}

Let $\mc U_{g,n}$ be the algebraic stack of Gorenstein reduced connected projective curves $C$ over $k$ of arithmetic genus $g$ endowed with $n$-marked disjoint points $\{p_1,\ldots, p_n\}$ with $\omega_C^{log}:=\omega_C(\sum p_i)$ ample. We call such a curve $(C,p_i)=(C,\{p_1,\ldots, p_n\})$ a \emph{$G$-stable} curve of type $(g,n)$. 

\begin{proposition}\footnote{This result was suggested to us by D. Gori.}\label{P:Ugn-finite}
    The stack $\mc U_{g,n}$ is of finite type over $k$.
\end{proposition}
\begin{proof}
Since the stack $\mc U_{g,n}$ is locally of finite type by \cite[\href{https://stacks.math.columbia.edu/tag/0DSS}{Tag 0DSS}]{stacks-project}, the statement reduces to show that $\mc U_{g,n}$ is quasi-compact. This can be shown as in  \cite[\href{https://stacks.math.columbia.edu/tag/0E9B}{Tag 0E9B}]{stacks-project} using that if $(C,p_i)$ is G-stable of type $(g,n)$, then we have that
\begin{enumerate}[(i)]
\item \label{prop1} $H^1(C,(\omega_C^{log})^{\otimes k})=0$ for any $k\geq 2$.
\item \label{prop2} $(\omega_C^{log})^{\otimes k}$ is very ample for any $k\geq 4$\footnote{With some extra work, it is possible to lower this to $3$, as in \cite[Thm. B]{Cat} and \cite[Thm. 1.8]{Knu2}, which is then sharp as shown in \cite[Thm. C]{Cat}. However, we will not prove this sharp result, since we will not need it.}.
\end{enumerate}
These properties are proved for unmarked G-stable curves (i.e. $n=0$) in \cite[Thm. B]{Cat} and for G-stable $n$-pointed curves with only double point singularities in \cite[Prop. 1.7]{Per1} (extending the proof of \cite[Thm. 1.8]{Knu2} for nodal singularities). 

We will see how to adapt the proof of \cite[Prop. 1.7]{Per1} to our general case. 

Part \ref{prop1}: by duality we have
$$
H^1(C,(\omega^{log})^{\otimes k})=H^0\left(C, (\omega^{log})^{\otimes (1-k)k}(-\sum_i p_i)\right)^\vee,
$$
and the last group is zero for every $k\geq 2$ since $(\omega^{log})^{\otimes (1-k)}(-\sum_i p_i)$ has negative degree on each irreducible component of $C$, because $\omega_C^{\log}$ is ample. 

In order to prove part \ref{prop2}, it is enough to show, using that $\omega_C^{log}$ is ample, that $(\omega_C^{log})^{\otimes k}$ is normally generated for any $k\geq 4$. This last property can be deduced, arguing as at the end of \cite[Prop. 1.7]{Per1}, from the Generalized Lemma of Castelnuovo and the following property:

$(\dagger)$: $(\omega_C^{log})^{\otimes k}$ is globally generated for any $k\geq 2$.

Indeed, arguing as in the proof of \cite[Thm. 1.7]{Per1}, for any $p\in C$, 
the surjectivity of the evaluation map
$$
H^0(C,(\omega_C^{log})^{\otimes k})\xrightarrow{\text{ev}_p} H^0(p,(\omega_C^{log})^{\otimes k}_{|p})
$$
follows from the vanishing 
\begin{equation}\label{E:van-pull}
 H^0\left(\wt C, \pi^*\left((\omega_C^{log})^{\otimes (1-k)}(-\sum_i p_i) \right)\right)=0,  
\end{equation}
where $\pi:\wt C\to C$ is the partial normalization of $C$ at $p\in C$. Now the vanishing \eqref{E:van-pull} follows from the fact that the line bundle $\pi^*\left((\omega_C^{log})^{\otimes (1-k)}(-\sum_i p_i) \right)$ has negative degree on each irreducible component of $\wt C$ because $\omega_C^{log}$ is ample. 
\end{proof}

The stack $\mc U_{g,n}$ contains, as an open substack, the integral DM stack $\ov{\mc M}_{g,n}$ of stable (=G-stable and nodal) $n$-pointed  curves of genus $g$ and we denote by $\mc U^{sm}_{g,n}$ the closure of $\ov{\mc M}_{g,n}$ in $\mc U_{g,n}$, endowed with reduced structure. Note that $\mc U^{sm}_{g,n}$ is the irreducible component of $\mc U_{g,n}$ consisting of smoothable curves, i.e. curves in $\mc U_{g,n}$ that are limit of $n$-pointed smooth curves.

We let $\pi:\mc C_{g,n}\to \mc  U_{g,n}$ to be the universal curve which comes endowed with $n$ disjoint sections $\{\sigma_i\}_{i=1}^n$ whose images $\{D_i:=\sigma_i(\mc U_{g,n})\}_{i=1}^n$ are disjoint Cartier divisors on $\mc C_{g,n}$.  Denote by $\omega_\pi$ the relative dualizing sheaf of $\pi$ and by $\wh{\omega}_{\pi}:=\omega_{\pi}(\sum_{i=1}^n D_i)$ the relative logartithmic dualizing sheaf of $\pi$ (both of them are line bundles by our assumption on $\mc U_{g,n}$). Consider the following line bundles on $\mc U_{g.n}$:
\begin{equation}\label{E:lamba}
\begin{aligned}
& \wh \lambda_n:=\det R\pi_*(\wh\omega_\pi^{\otimes n}) \text{ for any } n\geq 0, \\
& \lambda_n:=\det R\pi_*(\omega_\pi^{\otimes n}) \text{ for any } n\geq 0, \\
& \psi_i:=\sigma_i^*(\omega_{\pi}) \text{ for any } 1\leq i \leq n.
\end{aligned}
\end{equation}
The Knudsen-Mumford formula, together with relative duality for $\pi$, implies that (in additive notation)  
\begin{equation}\label{E:Mum-for}
\begin{aligned}
& \wh \lambda_n=\binom{n}{2} \lambda_{CM}+ \lambda,  \\
& \lambda_n=\binom{n}{2} (\lambda_{CM}-\psi)+ \lambda,  \\
\end{aligned}
\end{equation}
where $\lambda_{CM}$ (called the Chow-Mumford line bundle) and $\lambda$ (called the Hodge line bundle) are given by 
\begin{equation}\label{E:CM-Hod}
  \lambda_{CM}=\kappa_1:=\langle \wh \omega_{\pi}, \wh \omega_{\pi}\rangle_{\pi} \text{ and } \lambda:=\det R\pi_*(\mc O_\pi),
\end{equation}
where $\langle -,- \rangle_{\pi}$ is the Deligne pairing for the family $\pi$ and $\psi:=\sum_{i=1}^n \psi_i$. Using the above line bundles, we can now extend  the canonical divisor class $K$ and the total boundary class  $\delta$ from $\ov{\mc M}_{g,n}$ to $\mc U_{g,n}$ by the formulas (see \cite[Chap. 13, Sec. 7]{GAC2})
\begin{equation}\label{E:K-delta}
  \delta-\psi:=12\lambda-\lambda_{CM} \text{ and } K+\psi:=2\lambda_{CM}-11\lambda.
\end{equation}

In what follows, we will often consider the following $\bb Q$-divisor class (for $\alpha\in \bb Q\cap [0,1]$)
\begin{equation}\label{E:Kdelta}
K+\psi+\alpha(\delta-\psi)=(2-\alpha)\lambda_{CM}+(12\alpha-11)\lambda=(2-\alpha)\wh \lambda_2+(13\alpha-13)\lambda.
\end{equation}

\subsection{$\Theta$-stability for line bundles}
\label{Sub:Theta-stability}

Let $\mc U$ be an algebraic stack locally of finite type over $k$ with affine automorphism groups and let $\Theta$ be the quotient stack $[\bb A^1 / \bb G_m]$. 
A line bundle $\mc L$ on $\mc U$ defines a numerical invariant 
\begin{equation}\label{E:mu-L}
   \begin{aligned}
       \wei_{-}(\mc L): \Map(\Theta,\mc U)& \longrightarrow \Pic B\bb G_m  \xrightarrow[\cong]{\wei} \bb Z\\
       T & \mapsto \mc L_{T(0)}  \mapsto \wei(\mc L_{T(0)}),
   \end{aligned} 
\end{equation} 
where last isomorphism $\wei$ is normalized according to the following convention: $\wei^{-1}(1)$ is the one dimensional vector space $\langle e \rangle$ linearized by the action $\lambda \cdot e:=\lambda^{-1}e$ (see also \cite[Sec. 3.3]{AHL} for a discussion of this sign convention). Note that $\wei_T(\mc L_1\otimes \mc L_2)=\wei_T(\mc L_1)+\wei_T(\mc L_2)$ for any two line bundles $\mc L_1$ and $\mc L_2$ on $\mc U$. 
Following \cite{HL-structure} and \cite{Hei}, we say that $p\in \mc U(k)$ is \emph{$\Theta$-semistable} with respect to a line bundle $\mc L$ if for any map $T: \Theta \to \mc U$ with $T(1) \cong p$ we have $\wei_T(\mc L) \ge 0$. 

We now specialize the above definition to our setting.

\begin{definition}\label{D:Ugalpha}
Let $\alpha \in \mathbb Q \cap [0,1]$. 
\begin{enumerate}
\item We say that a curve $C \in \mc U_{g,n}(k)$ is $\alpha$-semistable if it is $\Theta$-semistable for the $\bb Q$-line bundle $K+\psi+\alpha (\delta-\psi)$.
\item We denote by $\mc U_{g,n}(\alpha)$ the locus of curves $C\in \mc U_{g,n}$ that are $\alpha$-semistable and smoothable. 
\end{enumerate}
\end{definition}

Note that a map $T:\Theta\to \mc U_{g,n}$ such that $T(1)=(C,p_i)\in \mc U_{g,n}(k)$ is equivalent to a $n$-pointed test configuration $f:(\mc C,\sigma_i)\to \bb A^1$ for $(C,p_i)$ such that $\mc C_0\in \mc U_{g,n}(k)$. In particular, we have that $\omega_{\mc C/\bb A^1}(\sum_ i \sigma_i)$ is relatively ample.

\begin{lemma} \label{l:weight-test}
Consider a map $T: \Theta \to \mc U_g(k)$ such that $T(1)=(C,p_i)$ with induced test configuration $f:(\mc C, \sigma_i)\to \bb A^1$ and consider the canonical compactification $\ov f: (\ov{\mc C},\sigma_i) \to \bb P^1$ which is trivial at infinity.
Consider the following line bundles or divisor classes on $\bb P^1$:
$$\begin{aligned}
& \lambda_q(\ov f):=\det \ov f_*\mathcal{O}_{\ov{\mc C}}(qK_{\ov{\mc C}/\bb P^1}), \\
& \wh{\lambda}_q(\ov f):=\det \ov f_*\mathcal{O}_{\ov{\mc C}}(q(K_{\ov{\mc C}/\bb P^1}+\sum_i \sigma_i)), \\
& \lambda_{CM}(\ov f)=\ov f_*((K_{\ov{\mc X}/\bb P^1}+\sum_i \sigma_i)^2),
\end{aligned}
$$
all of which have an induced canonical $\bb G_m$-action.  We have that
$$
\begin{aligned}
& \wei_T(\lambda_q)= \deg \lambda_q(\ov f), \\
& \wei_T(\wh{\lambda}_q)= \deg \wh{\lambda}_q(\ov f), \\
& \wei_T(\lambda_{CM})= \deg \lambda_{CM}(\ov f)=	(K_{\ov{\mc C} / \bb P^1}+\sum_i \sigma_i)^2.
\end{aligned}
$$	
\end{lemma}
\begin{proof}
The first two equalities follow from the fact that if $E$ is a  vector bundle $E$ on $\mathbb{P}^1$ with a $\bb G_m$-action and with trivial weight at $\infty$, we have $\wei(\det(E))=\deg E$.
The last equality follows from the previous fact together with  the fact that 
$$
\deg(\langle \omega_{\ov f}(\sum_i \sigma_i),\omega_{\ov f}(\sum_i \sigma_i)\rangle_{\ov f})= \deg \ov f_*((K_{\ov{\mc X}/\bb P^1}+\sum_i \sigma_i)^2)= (K_{\ov{\mc C} / \bb P^1}+\sum_i \sigma_i)^2.
$$

\end{proof}

\section{Stability via slope inequalities}\label{S:stability-via-slope}

 The aim of this section is to apply slope inequalities to prove the following

\begin{theorem}\label{T:stab}
Let $(C, p_i)$ be a stable $n$-pointed nodal curve of genus $g$. Then $(C,p_i)$ is $\alpha$-semistable for $9/11\leq \alpha \le 1$.
\end{theorem}
\begin{proof}
Let $T:\Theta\to \mc U_{g,n}$ such that $T(1)=(C,p_i)$. 
Let $f:(\mc C, \sigma_i)\to \bb A^1$ be the induced test configuration and consider the canonical compactification $\ov f: (\ov{\mc C},\sigma_i) \to \bb P^1$ which is trivial at infinity.

By Lemma \ref{l:weight-test} and \eqref{E:Kdelta}, we compute 
\begin{equation}\label{E:weiT}
\begin{aligned}
& \wei_T(K+\psi+\alpha(\delta-\psi))=\wei_T((2-\alpha)\lambda_{CM}+(12\alpha-11)\lambda)=\\
& =(2-\alpha)\deg \lambda_{CM}(\ov f)+(12\alpha-11)\wh{\lambda}_1(\ov f)=\\
& =(2-\alpha)(K_{\ov{\mc C}/\bb P^1}+\sum_i \sigma_i)^2+(12\alpha-11)\deg \ov f_*\mathcal{O}_{\ov{\mc C}}(K_{\ov{\mc C}/\bb P^1}+\sum_i \sigma_i)
\end{aligned}
\end{equation}

Observe that $\ov{\mc C}$ is demi-normal by Lemma \ref{l:singularities-test} since the generic fiber (which is $C\times_{k} \Spec k(\bb P^1)$) is nodal, i.e. demi-normal, and $\mc C_0$ is reduced since $T(0)=(\mc C_0,\sigma_i(0))\in \mc U_{g,n}$ by assumption. Hence, the family $\ov f:(\ov{\mc C},\Delta:=\sum_i \sigma_i)\to \bb P^1$ is a generic slc family as in \cite[Setup 5.1]{CTV23} since $K_{\ov{\mc C}/\bb P^1}+\Delta$ is Cartier and the general fiber $(C,\Delta_C=\sum_i p_i)$ is semi-log canonical. Since $K_{\ov{\mc C}/\bb P^1}+\Delta$ is $\ov f$-ample (because it come from a morphism $T:\Theta \to \mc U_{g,n}$), we can apply \cite[Cor. 5.3(2)]{CTV23} in order to deduce that $K_{\ov{\mc C}/\bb P^1}+\Delta$ is nef. We now can apply the slope inequality \cite[Thm. 5.6(3)]{CTV23} with $q=1$ (using that $(K_{\ov{\mc C}/\bb P^1}+\Delta)_C=K_C+\sum_i p_i$ is Cartier and ample because $(C,p_i)\in \mc U_{g,n}$) in order to get 
\begin{equation}\label{E:slo-ine}
   (K_{\ov{\mc C}/\bb P^1}+\Delta)^2\geq  \deg \ov f_*\mathcal{O}_{\ov{\mc C}}(K_{\ov{\mc C}/\bb P^1}+\Delta).
\end{equation}
By combining \eqref{E:weiT} and \eqref{E:slo-ine}, we get that 
\begin{equation}\label{E:wei-fin}
 \wei_T(K+\psi+\alpha(\delta-\psi))\geq (11\alpha-9) \deg \ov f_*\mathcal{O}_{\ov{\mc C}}(K_{\ov{\mc C}/\bb P^1}+\Delta),
\end{equation}
which is non-negative since $\alpha \geq \frac{9}{11}$ and $f_*\mathcal{O}_{\ov{\mc C}}(K_{\ov{\mc C}/\bb P^1}+\Delta)$ is nef (see for instance \cite[Thm. 1.10]{Fujino} or \cite[Thm. 1.1]{CPT}). Since this is true for any $T:\Theta\to \mc U_{g,n}$ such that $T(1)=(C,p_i)$, we deduce that $(C,p_i)$ is $\alpha$-semistable for any $9/11 \leq \alpha \leq 1$.
\end{proof}	

\medskip

We conclude this section with the proof of Theorem \ref{T:Xiao}\eqref{T:Xiao1}. 

\begin{proof}[Proof of Theorem \ref{T:Xiao}\eqref{T:Xiao1}]
We follow the strategy of the proof of Theorem \ref{T:stab}, but applying a different slope inequality. 
Let $T:\Theta\to \mc U_{g}$ such that $T(1)=C$. 
Let $f:\mc C\to \bb A^1$ be the induced test configuration and consider the canonical compactification $\ov f: \ov{\mc C} \to \bb P^1$ which is trivial at infinity. Then, as in Equation \ref{E:weiT}, we have

$$
\wei_T(K+\alpha\delta) = (2-\alpha)K_{\ov{\mc C}/\bb P^1}^2+(12\alpha-11)\deg \ov f_*\mathcal{O}_{\ov{\mc C}}(K_{\ov{\mc C}/\bb P^1}).
$$

Observe that $\ov{\mc C}$ is normal by Lemma \ref{l:singularities-test}, and  $K_{\ov{\mc C}/\bb P^1}$ and $f_*\mathcal{O}_{\ov{\mc C}}(K_{\ov{\mc C}/\bb P^1})$ are nef as explained in the proof of Theorem \ref{T:stab}.  By \cite[Theorem E]{CTV23} (which generalizes \cite{Xiao} to non-smooth surfaces and \cite{CH88} to non-nodal special fibers), we know that 
$$
K_{\ov{\mc C}/\bb P^1}^2 \ge 4\frac{g-1}{g} \deg \ov f_*\mathcal{O}_{\ov{\mc C}}(K_{\ov{\mc C}/\bb P^1}).
$$
Hence 
$$
\wei_T(K+\alpha\delta) \ge \frac{\alpha(8g+4)-(3g+8)}{g}\deg f_*\mathcal{O}_{\ov{\mc C}}(K_{\ov{\mc C}/\bb P^1}),
$$
which is non-negative if 
$
\displaystyle \alpha \ge \frac{3g+8}{8g+4},
$
proving the theorem.  
\end{proof}

\section{Isotrivial specializations of Gorenstein singularities}

\subsection{Local case}\label{Sub:local}

The aim of this subsection is to show that any Gorenstein curve singularity can be isotrivially degenerated to a Gorenstein curve singularity with a $\bb G_m$-action. 

We will start by recalling some facts about Gorenstein curve singularities (resp. with  $\bb G_m$-action), following \cite[VIII]{AK}, \cite[Chap. IV]{Ser} (resp. \cite[Sec. 2]{AFS16}, \cite[Sec. 2]{CY}).

A \emph{curve singularity} (over a fixed algebraically closed ground field $k$) is a complete local reduced Noetherian $k$-algebra $(\mc{O},\mathfrak m )=\mc O$ of dimension one, and hence it 
can be written in the form
$$\mc O=\frac{k[[x_1,\ldots, x_n]]}{I}$$
where $I\subseteq (x_1,\ldots, x_n)^2$. In particular, if $C$ is a curve over $k$ and $q\in C$, then $\wh{\mc O}_{C,q}$ is a curve singularity. 

The \emph{normalization} of $\mc O$ is isomorphic to 
$$
\mc O^n\cong k[[t_1]]\times \ldots \times k[[t_b]],
$$
where $b=b(\mc O)$ is the number of \emph{branches} of $\mc O$. The colength  of $\mc O\subseteq {\mc O}^n$ is called the \emph{delta invariant} of $\mc O$, and it is denoted by $\delta=\delta(\mc O)$.

Since $\mc O$ is local, then its \emph{seminormalization} is isomorphic to 
$$
\mc O^{sn}\cong k[[t_1]]\times_k \ldots \times_k k[[t_b]].
$$
The colength  of $\mc O\subseteq {\mc O}^{sn}$ is called the \emph{genus} of $\mc O$,  denoted by $g=g(\mc O)$, and it is equal to $g=\delta-b+1$. Note that:
\begin{itemize}
    \item $\delta(\mc O)=0$ if and only if $\mc O$ is normal (and connected), or equivalently regular (and connected), i.e. $\mc O=\mc O^n\cong k[[t_1]]$;
     \item $g(\mc O)=0$ if and only if $\mc O$ is seminormal (and connected), i.e. $\mc O=\mc O^{sn}\cong k[[t_1]]\times_k \ldots \times_k k[[t_b]]$.
\end{itemize}

A curve singularity is called \emph{decomposable} if, up to permuting the variables $\{t_1,\ldots,t_b\}$, there exist two curve singularities $\mc O_1\subseteq \mc O_1^n=k[[t_1]]\times \ldots k[[t_{b'}]]$ and $\mc O_2\subseteq \mc O_2^n=k[[t_{b'+1}]]\times \ldots k[[t_{b}]]$ for some $1\leq b'<b$ such that $\mc O=\mc O_1\times_k \mc O_2$. Otherwise, the curve singularity is called \emph{indecomposable}.

The \emph{conductor ideal}
$
{\mathfrak{c}}(\mc O)= {\mathfrak{c}}(\mc O\subseteq \mc O^n):=\{f \in \mc O^n\: : f\cdot \mc O^n\subseteq \mc O\}
$
of the extension $\mc O\subseteq \mc O^n$, which is also the biggest ideal of $\mc O^n$ which is contained in $\mc O$, is equal to 
$${\mathfrak{c}}(\mc O\subseteq \mc O^n)=(t_1^{c_1})\times \ldots \times (t_b^{c_b}),$$ for a $b$-tuple $\un{c}=\un c(\mc O)=(c_1,\ldots, c_b)\in \bb N^b$ (called the \emph{multi-conductance} of $\mc O$). The \emph{conductance} of $\mc O$ is defined to be $c=c(\mc O)=|\un c(\mc O)|:=\sum_i c_i$.

\begin{lemma}\label{L:ci}
The multi-conductance $\un c(O)=(c_1,\ldots, c_b)$ of a non-regular curve singularity $\mc O$ satisfies
\begin{enumerate}
    \item \label{L:ci1} $c_i\geq 1$ for every $1\leq i \leq b$;
    \item \label{L:ci2} $c_i=1$ if and only if $b\geq 2$ and there exists a curve singularity $\mc O'\subseteq (\mc O')^{sn}=k[[t_1]]\times \ldots \times \wh{k[[t_{i}]]}\times \ldots k[[t_b]]$ such that 
    $$\mc O=k[[t_i]]\times_k \mc O'.$$
    In particular, if $\mc O$ is indecomposable, then $c_i\geq 2$ for every $1\leq i \leq b$.
\end{enumerate}
\end{lemma}
\begin{proof}
If $b=1$ then $c(\mc O)=c_1\geq 2$ since $k+t_1k[[t_1]]=k[[t_1]]$ and $\mc O$ is not regular.
Suppose now that $b\geq 2$. Up to permuting the variables $\{t_1,\ldots, t_b\}$, we can assume that $i=1$. 

In order to prove part \eqref{L:ci1}, suppose by contradiction that $c_1=0$. Then $(1,0,\ldots,0)\in {\mathfrak{c}}(\mc O\subseteq \mc O^n)$ which would imply that 
$$(1,0,\ldots, 0) \mc O^n=k[[t_1]]\times 0\times \ldots \times 0\subseteq \mc O$$
and this violates the fact that $\mc O$ is local. 

Let us prove part \eqref{L:ci2}. The if condition follows from the straightforward formula
$$
{\mathfrak c}(k[[t_1]]\times_k \mc O')={\mathfrak c}(k[[t_1]])\times {\mathfrak c}(\mc O')=(t_1)\times {\mathfrak c}(\mc O').
$$
Conversely, assume that $c_1=1$. Then, using that $(t_1,0,\ldots, 0)\in {\mathfrak c}(\mc O)$, we have 
\begin{equation}\label{E:inclu1}
(t_1k[[t_1]],0,\ldots,0)=(t_1,0,\ldots,0)\mc O^n\subset \mc O.
\end{equation}
Consider the projection onto the last $b-1$ factors
$$\pi:k[[t_1]]\times_k k[[t_2]]\times_k \ldots \times_k k[[t_b]]\to k[[t_2]]\times_k \ldots \times_k k[[t_b]]$$
and let $\mc O':=\pi(\mc O)$ which is clearly a curve singularity with $(\mc O')^n=k[[t_2]]\times \ldots \times k[[t_b]].$ By construction, we have that 
\begin{equation}\label{E:inclu2}
    \mc O\subseteq k[[t_1]]\times_k \mc O'.
    \end{equation}
Moreover, since the kernel of $\pi$ is equal to $(t_1k[[t_1]],0,\ldots, 0)$ and it is contained in $\mc O$ by \eqref{E:inclu1}, we must have equality in \eqref{E:inclu2} and we are done. 
\end{proof}

The \emph{dualizing module} $\omega_{\mc O}$ can be described as the $\mc O$-submodule of the module of rational differentials on $\mc O^n$ 
$$
\Omega^{rat}_{\mc O^n}:=\left\{\omega=(r_1dt_1,\ldots,r_bdt_b)\: : r_i\in k((t_i)) \text{ for every } 1\leq i \leq b \right\}
$$
given by 
\begin{equation}\label{E:Ros}
  \omega_{\mc O}:=\{\omega\in \Omega^{rat}_{\mc O^n} : \sum_{i=1}^n \Res_{t_i}(g\omega)=0 \text{ for every } g\in \mc O\}.   
\end{equation}
The rational differentials appearing in \eqref{E:Ros} are called Rosenlicht differentials of $\mc O$.

\begin{lemma}\label{L:omega}
 Let $\mc O$ be a curve singularity.
 \begin{enumerate}
     \item \label{L:omega1} We have that 
     $$
     \omega_{\mc O}\subseteq k[[t_1]]\frac{dt_1}{t_1^{c_1}}\oplus \ldots \oplus k[[t_b]]\frac{dt_b}{t_b^{c_b}}.
     $$
     \item \label{L:omega2} The module $\omega_{\mc O}$ contains an element of the form 
     $$
     \omega=\left(u_1\frac{dt_1}{t_1^{c_1}}, \ldots, u_b \frac{dt_b}{t_b^{c_b}}\right), \text{ where } u_i\in k[[t_i]]^*.
     $$
 \end{enumerate}
\end{lemma}
\begin{proof}
Part \eqref{L:omega1}: if the statement is false, then, up to permuting the variables $\{t_1,\ldots, t_b\}$, we can find an element   $\omega_o=(r_1dt_1,\ldots,r_bdt_b)\in \omega_{\mc O}$ such that $\val r_1 < c_1$. This leads to a contradiction since the element $g_o:=(t_1^{-1-\val(r_1)}, 0,\ldots, 0)$ belongs to $(t_1^{c_1})\times \ldots \times (t_b^{c_b})={\mathfrak{c}}(\mc O\subseteq \mc O^n)\subseteq \mc O$ and we have that  
$$
\Res(g_o\omega_o)=\Res_{t_1}(t_1^{-1-\val(r_1)}r_1dt_1)\neq 0.
$$

Part \eqref{L:omega2}: see the proof of \cite[Prop. 7]{Ser}.
\end{proof}

The curve singularity $\mc O$ is said to be \emph{Gorenstein} if $\omega_{\mc O}$ is free of rank one. In the following Lemma, we collect some property of Gorenstein singularities that we will need later on.

\begin{lemma}\label{L:Gor}
 Let $\mc O$ be a curve singularity.
 \begin{enumerate}
     \item \label{L:Gor1} We have that $c(\mc O)\leq 2\delta(\mc O)$ with equality if and only if $\mc O$ is Gorenstein.
     \item \label{L:Gor2} $\mc O$ is Gorenstein and decomposable if and only if $\mc O$ is a node.
     \item \label{L:Gor3} If $\mc O$ is Gorenstein, then $\omega_{\mc O}$ is generated by an element of the form 
     $$
     \omega=\left(u_1\frac{dt_1}{t_1^{c_1}}, \ldots, u_b \frac{dt_b}{t_b^{c_b}}\right), \text{ where } u_i\in k[[t_i]]^*.
     $$
     \item \label{L:Gor4} If $\mc O$ is Gorenstein, then $$\omega_{\mc O}\otimes_{\mc O} \mc O^n=  k[[t_1]]\frac{dt_1}{t_1^{c_1}}\oplus \ldots \oplus k[[t_b]]\frac{dt_b}{t_b^{c_b}}.$$
 \end{enumerate}
\end{lemma}
\begin{proof}
    Part \eqref{L:Gor1}: see \cite[Ch. IV, Prop. 7]{Ser}.
    Part \eqref{L:Gor2}: see \cite[Prop. 2.1]{AFS16}.
    Part \eqref{L:Gor3} follows by Lemma \ref{L:omega}.
    Part \eqref{L:Gor4} follows by combining part \eqref{L:Gor3} and Lemma \ref{L:omega}\eqref{L:omega1}.
\end{proof}

Suppose now that $\mc O$ is a \emph{Gorenstein curve singularity with a $\bb G_m$-action}.
Then the action of $\bb G_m$ extends to $\mc O^n$ in such a way that the inclusion $\mc O\subset \mc O^n$ is $\bb G_m$-equivariant. From Lemma \ref{L:Gor}\eqref{L:Gor3}, we deduce that dualizing module $\omega_{\mc O}$ is freely generated by a Rosenlicht differential of the form
$$
\omega_0=\left(K_1\frac{dt_1}{t_1^{c_1}},\ldots,K_b\frac{dt_b}{t_b^{c_b}}  \right), \text{ where } K_i\in k^*.
$$
If, furthermore, $\mc O$ is non-nodal (which implies that $c_i\geq 2$ for any $i$, by Lemmas \ref{L:ci}\eqref{L:ci2} and \ref{L:Gor}\eqref{L:Gor2}), then the action of $\bb G_m$ on $\mc O^n$ is given by (up to scaling by a common factor and changing signs)
\begin{equation}\label{E:Gm-O}
\lambda\cdot (t_1,\ldots, t_b):=(\lambda^{-a_1}t_1, \ldots, \lambda^{-a_b}t_b),
\end{equation}
where 

$\star$ $l:=l(\mc O)=\lcm (c_1-1,\ldots, c_b-1)\in \bb N_{>0}$;

$\star$ $\un a:=\un a(\mc O)=(a_1:=\frac{l}{c_1-1},\ldots, a_b:=\frac{l}{c_b-1})$ with $\gcd (a_1,\ldots, a_b)=1$.

We will always normalize the $\mathbb G_m$-action on $\mc O$ in such a way that it is given by \eqref{E:Gm-O}, and in particular it has negative weights. Given such a Gorenstein curve singularity $\mc O$ with $\mathbb G_m$-action, we will denote by $\Opp$ the opposite curve singularity, i.e. the curve singularity where the action on the normalization $(\Opp)^n$ is given by  
\begin{equation}\label{E:Gm-opp}
\lambda\cdot (t_1,\ldots, t_b):=(\lambda^{a_1}t_1, \ldots, \lambda^{a_b}t_b).
\end{equation}

\vspace{0.1cm}

Curves singularities with fixed multi-conductance can be parametrizing via territories, as we now explain (following \cite{BGS}). 
Fix a $b$-tuple $\un c=(c_1,\ldots, c_b)\in \bb N_{\geq 2}^b$ such that $c=|\un c|=2\delta$ for some $\delta\geq 1$ and set $g:=\delta-b+1$. As explained in \cite[Sec. 3]{BGS}, there exists a projective scheme $\Ter_{\un c}$ (which is denoted by $\Ter_{A_{\un c}^+}^g$ in \cite[Page 17]{BGS}) whose $k$-points are 
$$\Ter_{\un c}(k):=\{B\subset A_{\un c}:=\frac{k[t_1]}{(t_1^{c_1})}\times \ldots \times \frac{k[t_b]}{(t_b^{c_b})}: \: B \text{ is local and of codimension } \delta\},$$
and an open subscheme $\Ter_{\un c}^o\subset \Ter_{\un c}$ (which is denoted by $\Ter_{\mc S}(g,\un c)$ in \cite[Def. 3.4]{BGS}) whose $k$-points are 
$$\Ter^o_{\un c}(k):=\{B\in \Ter_{\un c}: \: {\mathfrak c}(B\subset A_{\un c})=(0)\}.$$
Given $B\in \Ter_{\un c}$, we can produce a curve singularity $\mc O(B)$  by taking the following fiber product 
$$
\begin{tikzcd}
  \mc O(B) \arrow[hook, d] \arrow[r] \arrow[dr, phantom, "\lrcorner", very near start]
  & B \arrow[hook, d] \\
   k[[t_1]]\times \ldots \times k[[t_b]] \arrow["\pi", twoheadrightarrow, r]   & A_{\un c} \\
\end{tikzcd}
$$
In this way, we get a surjective map 
\begin{equation}\label{E:Ter-O}
\begin{aligned}
   \Ter_{\un c}(k) & \twoheadrightarrow \left\{\begin{aligned}\text{Curve singularities with delta invariant } \delta \\ \text{and conductor containing } (t_1^{c_1})\times \ldots \times (t_b^{c_b})\end{aligned}\right\}\\
   B\subset A_{\un c} & \mapsto \mc O(B)
\end{aligned}
\end{equation}

\begin{lemma}\label{L:O-Gor}
For a given $B\in \Ter_{\un c}(k)$, we have that
$$
\mc O(B) \text{ is Gorenstein } \Leftrightarrow \un c(\mc O(B))=\un c \Leftrightarrow B\in \Ter^o_{\un c}(k).
$$
\end{lemma}
\begin{proof}
Indeed, since $2\delta=|\un c|$ by assumption and $(t_1^{c_1})\times \ldots \times (t_b^{c_b}) \subseteq \mathfrak c(\mc O(B)\subset \mc O(B)^n)$, 
Lemma \ref{L:Gor}\eqref{L:Gor1} implies that $\mc O(B)$ is Gorenstein if and only if  $(t_1^{c_1})\times \ldots \times (t_b^{c_b}) = \mathfrak c(\mc O(B)\subset \mc O(B)^n)$, or equivalently if and only if $\un c(\mc O(B))=\un c$. We now conclude using the following easily checked formula
$$\mathfrak c(\mc O(B)\subset k[[t_1]]\times \ldots \times k[[t_b]])=\pi^{-1}(\mathfrak c(B\subset A_{\un c})).$$
\end{proof}

We now show that we can always degenerate an element $B\in \Ter_{\un c}^o(k)$ to an element of $\Ter_{\un c}^o(k)$ with a $\bb G_m$-action. With this aim, define 
$$\begin{aligned} 
 \un a:=\left(a_1:=\frac{l}{c_1-1},\ldots, a_b:=\frac{l}{c_b-1}\right) \text{ where } l:=\lcm (c_1-1,\ldots, c_b-1)\\
\end{aligned}$$
and consider the action of $\bb G_m$ on $k[[t_1]]\times \ldots \times k[[t_b]]$ and on $A_{\un c}$ given by 
\begin{equation}\label{E:actGm}
\lambda\cdot (t_1,\ldots, t_b):=(\lambda^{-a_1}t_1, \ldots, \lambda^{-a_b}t_b),
\end{equation}
We get an induced action of $\bb G_m$ on $\Ter_{\un c}$ and on curve singularities with normalization isomorphic to $k[[t_1]]\times \ldots \times k[[t_b]]$, in such a way that \eqref{E:Ter-O} is equivariant.

\begin{proposition}\label{P:lim-Ter}
 Let $B\in \Ter_{\un c}^o$. Then $B_o=\lim_{\lambda\to 0} \lambda\cdot B\in \Ter_{\un c}^o$.   
\end{proposition}
\begin{proof}
 First of all, the limit $B_o=\lim_{\lambda\to 0} \lambda\cdot B $ exists in $\Ter_{\un c}$ since $\Ter_{\un c}$ is proper. Consider the associated limit of curve singularities $\mc O(B_o)=\lim_{\lambda\to 0} \lambda\cdot \mc O(B)$ inside $k[[t_1]]\times \ldots \times k[[t_b]]$. 
Since $B\in \Ter_{\un c}^o$ by assumption, Lemma \ref{L:O-Gor} implies that $\mc O(B)$ is Gorenstein with multi-conductance equal to $\un c$. Then Lemma \ref{L:Gor}\eqref{L:Gor3} implies that $\omega_{\mc O(B)}$ is generated by a Rosenlicht differential of the form 
$$
     \omega=\left(u_1\frac{dt_1}{t_1^{c_1}}, \ldots, u_b \frac{dt_b}{t_b^{c_b}}\right), \text{ where } u_i\in k[[t_i]]^*.
     $$
Since $\bb G_m$ acts on each $\frac{dt_i}{t_i^{c_i}}$ by 
$$\lambda\cdot \frac{dt_i}{t_i^{c_i}}=\lambda^{-a_i(1-c_i)} \frac{dt_i}{t_i^{c_i}}=\lambda^{l}\frac{dt_i}{t_i^{c_i}}, $$
the limit of $\omega$ as a Rosenlicht differential in  $\omega_{\mc O(B_o)}$ is equal to 
$$
\omega_o=\lim_{\lambda\to 0} \lambda\cdot \omega=\left(u_1(0)\frac{dt_1}{t_1^{c_1}}, \ldots, u_b(0) \frac{dt_b}{t_b^{c_b}}\right)\in \omega_{\mc O(B_o)}.
$$
Therefore, Lemma \ref{L:omega}\eqref{L:omega1} implies that the multi-conductance of $\mc O(B_o)$ satisfies
$$\un c(\mc O(B_o))\geq \un c,$$
with respect to the pointwise order relation. 
However, $\un c(\mc O(B_o))\leq \un c$ by construction (see \eqref{E:Ter-O}), and hence $\un c(\mc O(B_o))=\un c$. This implies, by Lemma \ref{L:O-Gor}, that $B_o\in \Ter_{\un c}^o$, as desired. 
\end{proof}

\begin{corollary}\label{C:lim-O}
Let $\mc O$ be a non-nodal Gorenstein curve singularity with normalization $\mc O^n=k[[t_1]]\times \ldots \times k[[t_b]]$.
Then $\mc O_o:=\lim_{\lambda\to 0} \lambda\cdot \mc O \subset k[[t_1]]\times \ldots \times k[[t_b]]$ exists and it is a Gorenstein curve singularity with $\bb G_m$-action and with $\un c(\mc O_o)=\un c(\mc O)$. 
\end{corollary}
\begin{proof}
 It follows from Proposition \ref{P:lim-Ter} and Lemma \ref{L:O-Gor}.  
\end{proof}

\subsection{Global case}\label{Sub:global}

The aim of this subsection is to globalize the results of \S \ref{Sub:local} by showing that any Gorenstein non-nodal projective curve can be isotrivially degenerated to a  Gorenstein projective curve with a $\bb G_m$-action. 

First of all, we need to recall a few facts from \cite{AFS16} on how to construct a global (affine or projective) Gorenstein curve with $\bb G_m$-action starting from a Gorenstein curve singularity with $\bb G_m$-action. More precisely, consider a non-nodal Gorenstein curve singularity $\mc O$ with $\bb G_m$-action and multi-conductance $\un c:=\un c(\mc O)=(c_1,\ldots, c_b)\in \bb N_{\geq 2}^b$ and normalize the action of $\bb G_m$ on $\mc O$ as in \eqref{E:Gm-O}. Then it is shown in \cite[Prop. 2.4]{AFS16} that there exists a unique affine pointed $k$-curve $X(\mc O)=(X(\mc O),p)$ (which we will call the \emph{$\mc O$-atom}) with a faithful action of $\bb G_m$ such that:

$\bullet$ The action of $\bb G_m$ is free outside $p$ and $X(\mc O)-p$ is smooth;

$\bullet$ The completion $\wh{\mc O}_{X(\mc O), p}$ is $\bb G_m$-equivariant isomorphic to $\mc O$.

The affine curve $X(\mc O)$ has $b:=b(\mc O)$ irreducible components $\{E_1,\ldots, E_b\}$, each of which passes through $p$ and having pointed normalization isomorphic to $(\bb A^1,0)$, and such that the action of $\bb G_m$ on each irreducible component $E_i$ is induced by the scalar multiplication of weight $-a_i$   on $(\bb A^1,0)$, where the $a_i$ are defined as in \eqref{E:Gm-O}. Therefore, we can compactify $X(\mc O)$ by adding a point $q_i$ at infinity on $E_i$: we obtain a connected projective $b(\mc O)$-pointed curve $\{\ov{X}(\mc O), q_i\}$, called the \emph{projective $\mc O$-atom}, together with a faithful action of $\bb G_m$ such that:

$\bullet$ $p\in X(\mc O)\subset \ov{X}(\mc O)$ is the unique singular point of $\ov{X}(\mc O)$ and the completion $\wh{\mc O}_{\ov{X}(\mc O), p}$ is $\bb G_m$-equivariant isomorphic to $\mc O$.

$\bullet$ $\ov{X}(\mc O)$ has $b$ irreducible components $\{\ov E_1,\ldots, \ov E_b\}$, each passing through $p$ and with $q_i\in E_i$, and such that the normalization $(\ov E_i,p,q_i)$ is isomorphic to $(\bb P^1, 0,\infty)$ and the action of $\bb G_m$ on $\ov E_i$ is induced by the scalar multiplication of weight $-a_i$ on $(\bb P^1, 0,\infty)$. 

$\bullet$ $\ov{X}(\mc O)$ has arithmetic genus equal to $g(\mc O)$.

$\bullet$ $\ov{X}(\mc O)$ is smoothable if and only if $\mc O$ is smoothable. 

Moreover, for any subset $ S\subsetneq \{1,\ldots, b\}$, we denote by $\ov X(\mc O)^S:=(\ov X(\mc O),\{q_i\}_{i \not \in S})$ the connected projective curve obtained from $(\ov X(\mc O),\{q_i\})$ by forgetting the marked points belonging to the branches corresponding to $S$ and keeping the other marked points. The curve $\ov X(\mc O)^S$ is called the \emph{$S$-dangling projective $\mc O$-atom} because the branches corresponding to $S$ "dangle", i.e. they do not have any marked point.  Observe that $\ov X(\mc O)^\emptyset=\ov X(\mc O)$.

\begin{lemma}\label{L:can-atom}
  The log-canonical line bundle of $\ov X(\mc O)^S:=(\ov X(\mc O),\{q_i\}_{i \not \in S})$ has multi-degree equal to
  $$
  \deg_{E_i}(\omega_{\ov{X}(\mc O)}(\sum_{i\not\in S} q_i))=
  \begin{cases}
      c_i-1 & \text{ if } i\not \in S,\\
      c_i-2 &\text{ if } i\in S.
  \end{cases}
  $$
  In particular, $\ov X(\mc O)^S:=(\ov X(\mc O),\{q_i\}_{i \not \in S})$ is G-stable  if and only if $c_i\geq 3$ for any $i\in S$.
\end{lemma}
\begin{proof}
  This follows by adjunction from Lemma \ref{L:Gor}\eqref{L:Gor4} and the fact that each $E_i$ is a rational curve.
\end{proof}

We denote by $\chi_i^{log}(\mc O^S)$ the character of $\bb G_m$ on 
$$
H^0(\ov X(\mc O)^S,\omega_{\ov X(\mc O)^S}(\sum_i q_i)^{\otimes i})
$$ 
for any $i\geq 1$, as in \cite[Sec. 3.1]{AFS16} and \cite[Sec. 2.3]{CY}. 
Note  that $\chi_i^{log}(\mc O^S)>0$  because of the action \eqref{E:Gm-O}.

One can perform similar constructions using the opposite curve singularity $\Opp$: we will denote the corresponding (projetive or dangling) atoms by $X(\Opp)$,  $\ov X(\Opp)$ and $\ov X(\Opp)^S$.
Similarly, we denote by $\chi_i^{log}(\Opp^S)$ the character of $\bb G_m$ on 
$$
H^0(\ov X(\Opp)^S,\omega_{\ov X(\Opp)^S}(\sum_i q_i)^{\otimes i})
$$ 
for any $i\geq 1$, and observe that we have 
$$
\chi_i^{log}(\Opp^S)=-\chi_i^{log}(\mc O^S).
$$

The following theorem is the main result of this section.

\begin{theorem}\label{T:deg-sing}
 Let $(C,p_i)\in \mc U_{g,n}$ having a singular point $q\in C$ which is not a node. Consider the 
  Gorenstein (non-nodal) curve singularity $\mc O$ with $\bb G_m$-action obtained from $\wh{\mc O}_{C,q}$ as in Corollary \ref{C:lim-O}. Then there exists  a test configuration $(\mc C,\sigma_i)\to \bb A^1$ in $\mc U_{g,n}$ for $(C,p_i)$ such that 
 \begin{enumerate}[(1)]
     \item \label{T:deg-sing1} $(\mc C_0,\sigma_i(0))$ is obtained from the pointed normalization $\nu: (C,p_i)^{\nu}\to (C,p_i)$ at $p$ and $\ov{X}(\mc O)$ by gluing nodally the marked points $\{q_i\}$ of $\ov{X}(\mc O)$ with the points of $\nu^{-1}(q)$, and then contracting the subcurves where the log-canonical bundle is not ample.
\item \label{T:deg-sing2} The action of $\bb G_m$ on $\mc C_0$, induced by the test configuration, is trivial on the image of $(C,p_i)^{\nu}$ and it is the standard action on the complementary subcurve, which is an $S$-dangling projective atom  $\ov{X}(\mc O)^S$.
 \end{enumerate}
 In particular, this test configuration defines a morphism $T:\Theta\to \mc U_{g,n}$ such $T(1)=(C,p_i)$.
\end{theorem}
Note that $\un c(\mc O)=\un c(\wh{\mc O}_{C,q})$, which implies that $\mc O$ and $\wh{\mc O}_{C,q}$ have the same delta-invariant, the same number of branches and the same genus. Moreover, the stable $S$-dangling projective atom $\ov X(\mc O)^S$ appearing in the central fiber $\mc C_0$ is such that the dangling branches $i\in S$ correspond to the branches of $\mc O$ which are contained in a rational $1$-pointed connected component of $(C^{\mu}, \nu^{-1}(q))$ (and this, by the G-stability of $(C,p_i)$, may occur only when $c_i\geq 3$).

\begin{proof}
Let $\nu:\wt C\to C$ be the normalization of $C$ at the singular point $q$ and let $\nu^{-1}(q)=\{q_1,\ldots, q_b\}$. 
Consider the conductor ideal sheaf $\mathfrak{c}(\nu):=\Ann_{\mc O_C}(\nu_*\mc O_{\wt C}/\mc O_C)$ of $\nu: \wt C \to C$. Since $\mathfrak{c}(\nu)$ is an ideal sheaf both in $\mc O_C$ and in $\mc O_{\wt C}$, then $\mathfrak{c}(\nu)$ defines two $0$-dimensional subschemes $D\hookrightarrow C$ and $\wt D\hookrightarrow \wt C$ supported, respectively, on $q$ and $\nu^{-1}(q)$. Since $(C,p)$ is Gorenstein and non-nodal, we have that (see Lemmas \ref{L:ci} and \ref{L:Gor}) $D$ has length equal to the delta invariant $\delta(C,q):=l(\nu_*\mc O_{\wt C}/\mc O_C)$ of the singularity $(C,q)$, while $\wt D=c_1q_1+\ldots +c_b q_b$ for $\un c(\wh{\mc O}_{C,q})=(c_1,\ldots, c_b)\in \bb N_{\geq 2}^b$ has length equal to $\sum_{i=1}^b c_i=2\delta(C,q)$.
As in subsection \ref{Sub:local}, we set
$$
\begin{aligned}
 & l:=\lcm(c_1-1,\ldots,c_b-1),\\
 & (a_1,\ldots,a_b):=\left(\frac{l}{c_1-1},\ldots,\frac{l}{c_b-1} \right).
\end{aligned}
$$
It is well-known that we have a push-out diagram (see e.g. \cite[Lemma 1.2]{vdW} or \cite[Lemma B.0.4]{Per})
\begin{equation}\label{E:push-out}
\begin{tikzcd}
  \wt D=\nu^{-1}(D) \arrow[hook, r]  \arrow[d]  & \wt C \arrow[d, "\nu"']   \\
      D  \arrow[hook, r]&  C \arrow[ul, phantom, "\ulcorner", very near start]
\end{tikzcd}
\end{equation}
Moreover, $\wt C$ is canonically isomorphic to the blow-up of $C$ at $D$ and $\wt D$ is the exceptional divisor of the blow-up (see \cite[Prop. B.0.5]{Per}).
In terms of the notation of subsection \ref{Sub:local}, we have that 
$$\begin{aligned}
& \wt D=\Spec A_{\un c} \text{ where } A_{\un c}:=\frac{k[t_1]}{(t_1^{c_1})}\times \ldots \times \frac{k[t_b]}{(t_b^{c_b})},\\
& D=\Spec B \text{ for some } B\in \Ter_{\un c}^o.
\end{aligned}$$ 

Consider now the morphism $b: \wt{\mc C}\to \wt C\times \bb A^1$ obtained by blowing-up  each point $(q_i,0)$ with weights $(a_i,0)$ and let $\wt{\mc D}\hookrightarrow \wt{\mc C}$ be the strict transform of $\wt D\times \bb A^1\hookrightarrow \wt C\times \bb A^1$.
Explicitly, if we denote by $t$ a base parameter of $\bb A^1$ at $0$ and by $x_i$ a fiber parameter at each of the point $q_i$, $b$ is obtained by blowing-up the ideal $(x_i,t^{a_i})$ for each $1\leq i \leq b$. Therefore, in the local coordinates $(x_i,t)$ around $(q_i,0)$, the surface $\wt{\mc C}$ is given the equation $\{u_ix_i-v_it^{a_i}=0\}$ inside $\bb A^2_{x_i,t}\times \bb P^1_{u_i,v_i}$. 
The exceptional divisor $E_i\cong \bb P^1$ of $b$ over $(q_i,0)$ has equation $\{x_i=t=0\}$ and it is attached to the strict transform of $\wt C\times \{0\}$ at the point $q_i^{\infty}:=(0,0,[0,1])$ where the surface $\wt{\mc C}$ has local equation $\displaystyle \left\{\frac{u_i}{v_i}x_i=t^{a_i}\right\}$. The divisor $\wt{\mc D}$ does not intersect the strict transform of $\wt C\times \{0\}$ while its restriction to the exceptional divisor $E_i$ is equal to $c_i[1,0]:=c_iq_i^0$. The group $\bb G_m$ acts on $\wt{\mc C}$ by acting locally via 
$$
\lambda\cdot t=\lambda t, \quad \lambda\cdot x_i=x_i, \quad \lambda \cdot u_i=\lambda^{a_i}u_i, \quad \lambda \cdot v_i=v_i.
$$
In particular, if we identify $(E_i,q_i^0,q_i^{\infty})$ with $(\bb P^1,0,\infty)$, the action of $\bb G_m$ on $E_i$ is induced by the scalar multiplication of weight $-a_i$ on $(\bb P^1,0,\infty)$. Summing up, we have a $\bb G_m$-equivariant diagram  
\begin{equation}\label{E:diagA}
\begin{tikzcd}
  \wt{\mc D} \arrow[hook, rr]  \arrow[dr]  & & \wt{\mc C} \arrow[dl]   \\
      &  \bb A^1 \arrow[ur, bend right=30, "\wt{\sigma}_i"']& 
\end{tikzcd}
\end{equation}
where the two vertical downward arrows are proper and flat, the $\sigma_i$ are the sections which are the inverse image via $(\nu\times \id)\circ b$ of the constant sections $p_i\times \bb A^1$ of $C\times \bb A^1\to \bb A^1$, and the restriction of the closed embedding $\wt{\mc D}\hookrightarrow \wt{\mc C}$ to $\bb A^1\setminus\{0\}$ is $\bb G_m$-equivariantly isomorphic to $\wt D\times (\bb A^1\setminus \{0\})\hookrightarrow \wt C\times (\bb A^1\setminus \{0\})$.

From the above construction, it follows that the $\bb G_m$-equivariant morphism $\wt {\mc D}\to \bb A^1$ is isomorphic to 
$$
\wt{\mc D}=\Spec (A_{\un c}\otimes k[t])\to \bb A^1=\Spec k[t],$$
where $A_{\un c}:=\frac{k[t_1]}{(t_1^{c_1})}\times \ldots \times \frac{k[t_1]}{(t_1^{c_1})}$ and the action of $\bb G_m$ is given by 
$$
\lambda \cdot t=\lambda t \text{ and } \lambda\cdot t_i=\lambda^{-a_i}t_i.
$$
We now apply Proposition \ref{P:lim-Ter}, in order to get  $\mc B\in \Ter_{\un c}^o([\bb A^1/\bb G_m])$ such that $\mc B_1=B$ and $\mc B_0=B_o$. By defining $\mc D=\Spec \mc B$, we get an extension of the diagram \eqref{E:diagA} to a diagram 
\begin{equation}\label{E:diagA-com}
\begin{tikzcd}
  \wt{\mc D} \arrow[hook, rr]  \arrow[d]  & & \wt{\mc C} \arrow[ddl]   \\
  \mc D \arrow[dr] && \\
      &  \bb A^1 \arrow[uur, bend right=30, "\wt{\sigma}_i"']& 
\end{tikzcd}
\end{equation}
such that 
\begin{enumerate}[(i)]
    \item \label{E:diagA-com1} the maps $\wt{\mc D}\to \mc D$ and $\mc D\to \bb A^1$ are proper and flat; 
    \item \label{E:diagA-com2} there exists an action of $\bb G_m$ on $\mc D$ making the diagram \eqref{E:diagA-com} $\bb G_m$-equivariant;
    \item \label{E:diagA-com3} the restriction of \eqref{E:diagA-com} to $ 1\in \bb A^1$ is isomorphic to 
    \begin{equation}\label{E:diagA-1}
\begin{tikzcd}
  \wt D \arrow[hook, rr]  \arrow[d]  & & \wt{C}\arrow[ddl]   \\
  D \arrow[dr] && \\
      &  1=\Spec k \arrow[uur, bend right=30, "\nu^{-1}(p_i)"']& 
\end{tikzcd}
\end{equation}
\item \label{E:diagA-com4} $\mc D_0=\Spec B_o$, as in Proposition \ref{P:lim-Ter}.

\end{enumerate}
Since $\wt{\mc D}, \wt{\mc C}, \mc D$ are flat and proper over $\bb A^1$ (by property \eqref{E:diagA-com1}), we can  form the push-out of the diagram \eqref{E:diagA-com} (by \cite[Lemma B.0.1]{Per}), in order to obtain
\begin{equation}\label{E:diagA-push}
\begin{tikzcd}
  \wt{\mc D} \arrow[hook, rr]  \arrow[d]  & & \wt{\mc C} \arrow[d]   \\
  \mc D \arrow[dr] \arrow[hook, rr] && \mc C \arrow[dl] \arrow[ull, phantom, "\ulcorner", very near start] \\
      &  \bb A^1 \arrow[ur, bend right=30, "\sigma_i"'] & 
\end{tikzcd}
\end{equation}
with the property that $\mc C\to \bb A^1$ is flat and proper, and the above diagram \eqref{E:diagA-push} commutes with base change over $\bb A^1$. Moreover, since the diagram \eqref{E:diagA-com} is $\bb G_m$-equivariant (by property \eqref{E:diagA-com2}), it follows from the universal property of the push-out that there exists an action of $\bb G_m$ on $\mc C$ making the diagram \eqref{E:diagA-push} $\bb G_m$-equivariant. Property \eqref{E:diagA-com3} implies that the restriction of \eqref{E:diagA-push} to $1\in \bb A^1$ is  isomorphic to the diagram \eqref{E:push-out}.
Finally, property \eqref{E:diagA-com4} and the construction of the diagram \eqref{E:diagA-push} imply that $\mc C_0$ satisfies the properties \eqref{T:deg-sing1} and \eqref{T:deg-sing2}.  
\end{proof}

\begin{corollary}\label{C:deg-sing} 
Consider the test configuration $T$ of Theorem \ref{T:deg-sing}. Then the weight of $K+\psi +\alpha (\delta-\psi)$ on $T$ is equal to 
$$
\wei_T(K+\psi +\alpha (\delta-\psi))=(\alpha-2)\chi_2^{\log}(\mc O^S)+(13-13\alpha) \chi_1^{\log}(\mc O^S)=:w_{\alpha}(\mc O^S).
$$
\end{corollary}
\begin{proof}
This follows from the description of the central fiber of the test configuration $T$ in Theorem \ref{T:deg-sing}, together with formula \ref{E:Kdelta} and our convention on the definition \eqref{E:mu-L} of $\wei_T(-)$. 
\end{proof}

\subsection{The Chen-Yu construction}\label{Sub:CY}

Our construction of a test configuration starting from a non-nodal Gorenstein singularity as in Theorem \ref{T:deg-sing} is the counterpart of a construction of Chen-Yu \cite{CY}, which we now generalize in the following

\begin{theorem}\label{T:deg-CY}
 Let $(\wh{\mc C},\wh{\sigma}_i)$ be a test configuration in $\mc U_{g,n}$ and let $\wh T:\Theta\to \mc U_{g,n}$ the associated morphism. Suppose that the central fiber $(C:=\wh{\mc C}_0, p_i:=\wh{\sigma}_i(0))$ has a connected subcurve $Z$ meeting the complementary subcurve $Z^c$ in the nodes $\{q_1,\ldots, q_r\}$ and containing the marked points $\{q_{r+1}, \ldots, q_{s+r}\}$. Assume that  $Z$ and $Z^c$ are $\bb G_m$-invariant and that the nodes $\{q_1,\ldots, q_r\}$ are fixed by $\bb G_m$. Pick a regular differential $\omega$ on $Z$ such that its divisor of zeros is equal to
 \begin{equation}\label{E:div-om}
 \div(\omega)=\sum_{i=1}^b m_i q_i
 \end{equation}
  for some smooth non-marked points $\{q_{r+s+1},\ldots, q_b\}$ of $(C,p_i)$ contained in $Z$ and with the property that $m_i\geq 1$ for $r+s+1\leq i\leq b$. Then there exists a  Gorenstein singularity $\mc O$ with $\mathbb G_m$-action and multi-conductance $\un c(\mc O)=(m_1+2,\ldots, m_b+2)$ and a test configuration $(\mc C,\sigma_i)\to \bb A^1$ in $\mc U_{g,n}$ such that 
 \begin{enumerate}[(1)]
 \item \label{T:deg-CY0} $(\wh{\mc C},\wh{\sigma}_i)_{|\bb A^1-\{0\}}$ is $\bb G_m$-equivariantly isomorphic to $(\mc C,\sigma_i)_{|\bb A^1-\{0\}}$.
     \item \label{T:deg-CY1} $(\mc C_0,\sigma_i(0))$ is obtained by gluing the complementary subcurve $Z^c$ nodally with $\ov X(\Opp)$ at the points $\{q_1,\ldots, q_r\}$. 
\item \label{T:deg-CY2} The action of $\bb G_m$ on $\mc C_0$, induced by the test configuration, coincides with the original action on $Z^c$ inside $C$ and it is the standard action on the complementary subcurve, which is an $S$-dangling projective atom  $\ov{X}(\Opp)^S$ where $S$ correspond to the branches $\{r+s+1,\ldots, b\}$.
 \end{enumerate}
 In particular, this test configuration defines a morphism $T:\Theta\to \mc U_{g,n}$ such $T(1)=(\wh{ \mc C}_1,\wh{\sigma}_i(1))=\wh T(1)$.
\end{theorem}
\begin{proof}
We are going to adapt the construction in \cite[Sec. 2.2]{CY} to our setting.

Set as usual
$$
\ell:=\operatorname{lcm}(m_1+1,\dots,m_b+1)\quad \text{ and }\quad  a_i:=\frac{\ell}{m_i+1} \text{ for any } 1\leq i \leq b. 
$$
Consider the morphism $\pi: \wt{\mc C}\to \wh{\mc C}$ which is obtained by blowing-up each point $(q_i,0)$ with weights $(a_i,0)$, endowed with the sections  $\wt \sigma_j$ which are the strict transform of the sections $\wh \sigma_j$. Explicitly, if we denote by $t$ a base parameter of $\bb A^1$ at $0$ and by $x_i$ a fiber parameter at each of the point $q_i$ along $Z$, $\pi$ is obtained by blowing-up the ideal $(x_i,t^{a_i})$ for each $1\leq i \leq b$. Therefore, in the local coordinates $(x_i,t)$ around $(q_i,0)$, the surface $\wt{\mc C}$ is given the equation $\{u_ix_i-v_it^{a_i}=0\}$ inside $\bb A^2_{x_i,t}\times \bb P^1_{u_i,v_i}$. 
The exceptional divisor $E_i\cong \bb P^1$ of $\pi$ over $(q_i,0)$ has equation $\{x_i=t=0\}$ and it is attached to the strict transform  of $Z$ (which we denote by $\wt Z$) at the point $q_i^{\infty}:=(0,0,[0,1])$ and to the strict transform of $Z^c$ (which we denote by $\wt{Z^c}$) at the point $q_i^0:=(0,0,[1,0])$. 
Since each point $q_i$ is fixed by $\bb G_m$, the action of $\bb G_m$ on $\wh{\mc C}$ lifts to an action on $\wt{\mc C}$, which is locally given by 
$$
\lambda\cdot t=\lambda t, \quad \lambda\cdot x_i=x_i, \quad \lambda \cdot u_i=\lambda^{a_i}u_i, \quad \lambda \cdot v_i=v_i.
$$
In particular, on the local coordinates $v_i/u_i$ at $q_i^0$ and $u_i/v_i$ at $q_i^{\infty}$, the action of $\bb G_m$ is given by 
$$
\lambda\cdot \frac{v_i}{u_i}=\lambda^{-a_i} \frac{v_i}{u_i}  \: \text{ and } \: \lambda\cdot \frac{u_i}{v_i}=\lambda^{a_i} \frac{u_i}{v_i}.
$$
Consider the line bundle on $\wt{\mc C}$
$$
\mc L=\omega_{\pi}\left(\sum_{j=1}^n\wt \sigma_i+l \wt Z \right).
$$
 Then we have the following formulas for the restrictions of $\mc L$:
\begin{itemize}
    \item $\mc L_{|\wt{\mc C}_t}=\omega_{\wt{\mc C}_t}(\sum_{j=i}^n \wt \sigma_j(t))$ for any $t\neq 0$, and hence it is ample.
    \item $\mc L_{|\wt{Z^c}}=\omega_{\wt{\mc C}_0}(\sum_{j=i}^n \wt \sigma_j(0))_{|\wt{Z^c}}$, and hence it is ample.
    \item $
         \mc L_{|\wt Z}=(\omega_{\wt{\mc C}_0})_{|\wt Z}\left(-\sum_{i=1}^b\frac{l}{a_i} q_i^{\infty}\right)\cong\omega_Z(-\sum_{i=1}^bm_iq_i)$, and hence it is trivial because of \eqref{E:div-om}.
    \item For any $1\leq i \leq r+s$, we have that $\mc L_{|E_i}=\omega_{E_i}(q_i^0+q_i^{\infty}+(m_i+1)q_i^{\infty})\cong \mc O_{\bb P^1}(m_i+1)$, which is ample.
    \item For any $r+s+1\leq i \leq b$, we have that $\mc L_{|E_i}=\omega_{E_i}(q_i^{\infty}+(m_i+1)q_i^{\infty})\cong \mc O_{\bb P^1}(m_i)$, which is ample since $m_i\geq 1$ if $r+s+1\leq i \leq b$ by assumption. 
\end{itemize}
The desired test configuration is given by 
$$
\wt{\mc C}\to \mc C:=\Proj_{\bb A^1} \left(\bigoplus_{k} (p_2)_*(\mc L^k)\right)\to \bb A^1
$$
where $p_2:\wt{\mc C}\xrightarrow{\pi} \wt{\mc C} \to \bb A^1$, endowed with the section $\sigma_j$ which are the images of $\wt \sigma_j$. 
\end{proof}

A special case of the above construction is contained in the following

\begin{corollary}\label{C:CY-const}
    Consider a curve $(C, p_i)\in \mc U_{g,n}$  that contains a connected subcurve $Z$ meeting the complementary subcurve $Z^c$ in the nodes $\{q_1,\ldots, q_r\}$ and containing the marked points $\{q_{r+1}, \ldots, q_{s+r}\}$. Pick a regular differential $\omega$ on $Z$ such that its divisor of zeros is equal to
 \begin{equation}\label{E:div-om}
 \div(\omega)=\sum_{i=1}^b m_i q_i
 \end{equation}
  for some smooth non-marked points $\{q_{r+s+1},\ldots, q_b\}$ of $(C,p_i)$ contained in $Z$ and with the property that $m_i\geq 1$ for $r+s+1\leq i\leq b$. Then there exists a  Gorenstein singularity $\mc O$ with $\mathbb G_m$-action and multi-conductance $\un c(\mc O)=(m_1+2,\ldots, m_b+2)$ and a test configuration $(\mc C,\sigma_i)\to \bb A^1$ in $\mc U_{g,n}$ such that 
 \begin{enumerate}[(1)]
 \item  The fiber of $(\mc C,\sigma_i)$ over $1\in \bb A^1$ is isomorphic to $(C,p_i)$.
     \item  $(\mc C_0,\sigma_i(0))$ is obtained by gluing the complementary subcurve $Z^c$ nodally with $\ov X(\Opp)$ at the points $\{q_1,\ldots, q_r\}$ and the $\bb G_m$-action  on $\mc C_0$, induced by the test configuration, is trivial on $Z^c$ and it is the standard action on the complementary subcurve, which is an $S$-dangling projective atom  $\ov{X}(\Opp)^S$ where $S$ correspond to the branches $\{r+s+1,\ldots, b\}$.
 \end{enumerate}
 In particular, this test configuration defines a morphism $T:\Theta\to \mc U_{g,n}$ such $T(1)=(C, p_i)$.
\end{corollary}
The special case of the above Corollary where $Z=C$ (or equivalently $r=0$) and $b=s=n$ is proved in \cite[Thm. 2.6(1) and 2.7]{CY}. 
\begin{proof}
   It follows by applying Theorem \ref{T:deg-CY} to the product test configuration $(\wh{\mc C},\wh{\sigma}_i)=(C\times \bb A^1,p_i\times \bb A^1)\to \bb A^1$.
\end{proof}

\begin{corollary}\label{C:deg-CY}
Consider the test configurations $\wh T$ and $T$ of Theorem \ref{T:deg-CY}. Then the difference of the weights of $K+\psi +\alpha (\delta-\psi)$ on $\wh T$ and $T$ is equal to 
$$
\wei_T(K+\psi +\alpha (\delta-\psi))-\wei_{\wh T}(K+\psi +\alpha (\delta-\psi))= $$
$$=(\alpha-2)\chi_2^{\log}(\Opp^S)+(13-13\alpha) \chi_1^{\log}(\Opp^S)=-w_{\alpha}(\mc O^S).
$$
\end{corollary}
\begin{proof}
This follows from the description of the central fiber of the test configuration $T$ in Theorem \ref{T:deg-CY}, together with formula \ref{E:Kdelta} and our convention on the definition \eqref{E:mu-L} of $\wei_T(-)$. 
\end{proof}

We conclude this section with the proof of Theorem \ref{T:Xiao}\eqref{T:Xiao2}. 

\begin{proof}[Proof of Theorem \ref{T:Xiao}\eqref{T:Xiao2}] 
Let $C$ be a hyperelliptic curve of genus $g$ (with no marked points). There are two different proofs of the fact that $C$ is not $\alpha$-semistable if $\alpha<\frac{3g+8}{8g+4}$. We give both proofs since we find a nice coincide that they give the same conclusion in the same range of $\alpha$.

\un{I proof:}  Pick a regular differential $\omega$ on $C$
 such that $\div(\omega)=(2g-2)p$, with $p$ a Weierstrass point on $C$. Then Corollary \ref{C:CY-const} gives rise to a test configuration $T_1:\Theta\to \mc U_g$ for $C$ such that $T_1(0)=\ov X(^{\opp}A_{2g})^{\{1\}}$ (see \cite[\S 3.1.1]{CY}). Hence, Corollary \ref{C:deg-CY} and the formulas of \cite[3.2.1]{AFS16} imply that 
 $$
 \wei_{T_1}(K+\alpha\delta)=(2-\alpha)\chi_2^{\log}(A_{2g}^{\{1\}})+(13\alpha-13)\chi_1^{\log}(A_{2g}^{\{1\}})=
 $$
 $$=(2-\alpha)(5g^2-4g)+(13\alpha-13)g^2=\alpha(8g^2+4g)-(3g^2+8g)<0,
 $$
where the last inequality follows by the assumption that $\displaystyle \alpha<\frac{3g^2+8g}{8g^2+4g}=\frac{3g+8}{8g+4}$. This proves that $C$ is not $\alpha$-semistable.

\un{II proof:}  Pick a regular differential $\omega$ on $C$
 such that $\div(\omega)=(g-1)(p_1+p_2)$, with $\{p_1,p_2\}$ a pair of conjugate non-Weierstrass points.
  Then Corollary \ref{C:CY-const} gives rise to a test configuration $T_2:\Theta\to \mc U_g$ for $C$ such that $T_2(0)=\ov X(^{\opp}A_{2g+1})^{\{1,2\}}$ (see \cite[\S 3.2]{CY}). Hence, Corollary \ref{C:deg-CY} and the formulas of \cite[3.2.2]{AFS16} imply that 
 $$
 \wei_{T_2}(K+\alpha\delta)=(2-\alpha)\chi_2^{\log}(A_{2g+1}^{\{1,2\}})+(13\alpha-13)\chi_1^{\log}(A_{2g+1}^{\{1,2\}})=
 $$
 $$=(2-\alpha)\frac{5g^2+g-4}{2}+(13\alpha-13)\frac{g^2+g}{2}=\alpha\frac{8g^2+12g+4}{2}-\frac{3g^2+11g+8}{2}<0,
 $$
where the last inequality follows by the assumption that $\displaystyle \alpha<\frac{3g^2+11g+8}{8g^2+12g+4}=\frac{(3g+8)(g+1)}{(8g+4)(g+1)}=\frac{3g+8}{8g+4}$. This proves that $C$ is not $\alpha$-semistable.
 \end{proof}

\section{Destabilizing singularities}

The aim of the first part of this Section is to classify the singularities of the curves that can appear in $\mc U_{g,n}(\alpha)$ for $\alpha$ sufficiently big, namely $\alpha \geq 5/9$.

In order to state our result, we need to recall the following definition of $\alpha$-invariant (see \cite[Sec. 2.3]{AFS16} and \cite[Sec. 4.2]{CY}). 
Given a non-nodal Gorenstein singularity $\mc O$ with $\bb G_m$-action and a subset $S$ of its branches, we define the $\alpha$-invariant of $\mc O^S$ to be 
\begin{equation}\label{E:alphaO}
\alpha(\mc O^S):=\frac{13\chi_1^{log}(\mc O^S)-2\chi_2^{log}(\mc O^S)}{13\chi_1^{log}(\mc O^S)-\chi_2^{log}(\mc O^S)}=\frac{13-2\frac{\chi_2^{log}(\mc O^S)}{\chi_1^{log}(\mc O^S)}}{13-\frac{\chi_2^{log}(\mc O^S)}{\chi_1^{log}(\mc O^S)}}
\end{equation}
with the convention that $\alpha(\mc O^S)$ is undefined if $13\chi_1^{log}(\mc O^S)-\chi_2^{log}(\mc O^S)=0$. 

\begin{remark}\label{R:alpha-inv}
   We have that:
    \begin{enumerate}[(i)]
        \item \label{R:alpha-inv1}
$\alpha(\mc O^S)\in 
\begin{cases}
    [0,1) & \text{ if } \frac{\chi_2^{log}(\mc O^S)}{\chi_1^{log}(\mc O^S)}\leq \frac{13}{2}, \\
    (-\infty,0) & \text{ if } \frac{13}{2}<\frac{\chi_2^{log}(\mc O^S)}{\chi_1^{log}(\mc O^S)}<13,\\
     (2,\infty) & \text{ if } 13<\frac{\chi_2^{log}(\mc O^S)}{\chi_1^{log}(\mc O^S)}.
\end{cases}$
\item  \label{R:alpha-inv2} $w_{\alpha}(\mc O^S)\leq 0\Leftrightarrow \begin{cases}
    \alpha\leq \alpha(\mc O^S) & \text{ if } \alpha(\mc O^S)>2, \\
    \alpha\geq \alpha(\mc O^S) & \text{ if } \alpha(\mc O^S)<1, \\
     \chi_2^{\log}(\mc O^S)=13\chi_1^{\log}(\mc O^S)\geq 0 & \text{ if } \alpha(\mc O^S) \text{ is undefined,}
\end{cases}$ 
\newline and the strict inequality (resp. equality) holds on the left if and only if it holds on the right. 
\end{enumerate}
\end{remark}

We now classify smoothable non-nodal dangling Gorenstein singularities with $\bb G_m$-action having $\alpha$-invariant at least $5/9$ (and at most $1$). 

\begin{lemma} \label{l:4}
Let $\ov X(\mc O)^S$ be an $S$-dangling projective $\mc O$-atom associated to a non-nodal Gorenstein smoothable curve singularity  $\mc O$ with a $\bb G_m$-action.
Assume that
$$
\chi_2^{log}(\mc O^S) \le 4 \chi_1^{log}(\mc O^S), \text{ or equivalently that } \alpha(\mc O^S)\in \left[\frac{5}{9},1\right).
$$

Then $\mc O$ is one of the singularities appearing in Table \ref{t:anyS}.

\begin{table}[ht]
\centering
\begin{tabular}{ |c|c|c|c|c| } 
\hline
$\alpha$-invariant &  $g,b,\underline c$ & $\chi_1$ & $\chi_2$ & Singularity type  \\  
\hline

$9/11$ 
& $(1,1,(2))$ 
& $1$ 
& $2$ 
& $A_2$ \\ \hline

$7/10$ 
& $(1,2,(2,2))$ 
& $1$ 
& $3$ 
& $A_3$ \\ \hline

$7/10$ 
& $(2,1,(4))$ 
& $4$ 
& $12$ 
& $A_4^{\{1\}}$ \\ \hline

$7/10$ 
& $(2,2,(3,3))$ 
& $3$ 
& $9$ 
& $A_5^{\{1,2\}}$ \\ \hline

$2/3$ 
& $(2,1,(4))$ 
& $4$ 
& $13$ 
& $A_4$ \\ \hline

$19/29$ 
& $(2,2,(3,3))$ 
& $3$ 
& $10$ 
& $A_5^{\{1\}}, A_5^{\{2\}}$ \\ \hline

$17/28$ 
& $(2,2,(3,3))$ 
& $3$ 
& $11$ 
& $A_5$ \\ \hline

$17/28$ 
& $(2,3,(3,3,2))$ 
& $3$ 
& $11$ 
& $D_6^{\{1,2\}}$ \\ \hline

$17/28$ 
& $(3,1,(6))$ 
& $9$ 
& $33$ 
& $A_6^{\{1\}}$ \\ \hline

$17/28$ 
& $(3,2,(4,4))$ 
& $6$ 
& $22$ 
& $A_7^{\{1,2\}}$ \\ \hline

$22/37$ 
& $(2,2,(4,2))$ 
& $4$ 
& $15$ 
& $D_5^{\{1\}}$ \\ \hline

$49/83$ 
& $(3,1,(6))$ 
& $9$ 
& $34$ 
& $A_6$ \\ \hline

$32/55$ 
& $(3,2,(4,4))$ 
& $6$ 
& $23$ 
& $A_7^{\{1\}}$, $A_7^{\{2\}}$ \\ \hline

$5/9$ 
& $(1,3,(2,2,2))$ 
& $1$ 
& $4$ 
& $D_4$ \\ \hline

$5/9$ 
& $(2,2,(4,2))$ 
& $4$ 
& $16$ 
& $D_5$ \\ \hline

$5/9$ 
& $(2,3,(3,3,2))$ 
& $3$ 
& $12$ 
& $D_6^{\{1\}}$, $D_6^{\{2\}}$ \\ \hline

$5/9$ 
& $(3,1,(6))$ 
& $8$ 
& $32$ 
& $E_6^{\{1\}}$ \\ \hline

$5/9$ 
& $(3,2,(4,4))$ 
& $6$ 
& $24$ 
& $A_7$ \\ \hline

$5/9$ 
& $(3,2,(5,3))$ 
& $7$ 
& $28$ 
& $E_7^{\{1,2\}}$ \\ \hline

$5/9$ 
& $(3,4,(3,3,3,3))$ 
& $4$ 
& $16$ 
& $\mc H(1,1,1,1)^{\{1,2,3,4\}}$ \\ \hline

$5/9$ 
& $(4,1,(8))$ 
& $16$ 
& $64$ 
& $A_8^{\{1\}}$ \\ \hline

$5/9$ 
& $(4,2,(5,5))$ 
& $10$ 
& $40$ 
& $A_9^{\{1,2\}}$ \\ \hline

\end{tabular}
\caption{Atoms with $\chi_2 \le 4 \chi_1$.}
\label{t:anyS}
\end{table}

\end{lemma}
 The case without dangling branches of the above Lemma follows from \cite[Thm. 1.5]{CY}, where the authors classify all non-nodal Gorenstein smoothable curve singularities with $\bb G_m$-action (without dangling branches) with $\alpha$-invariant belonging to $(3/8,1]$. However, allowing dangling branches, we get more singularities that in loc. cit. 

 Note that the extreme case where all the branches are dangling is somewhat special since it can only appear for $n=0$ and in genus equal to the genus of the singularity.

\begin{proof}
As usual, we set $g=g(\mc O)$, $b=b(\mc O)$, $l=l(\mc O)$ and $\un c(\mc O)=(c_1, \ldots,c_b)$  (see Subsection \ref{Sub:local}).
Then formulas \cite[Corollaries 3.3 and 3.6]{AFS16} and \cite[Prop. 4.1(iv)]{CY} give that 
\begin{equation}\label{E:chi1, chi2}
 \begin{sis}
& \chi_1^{\log}(\mc O^S)= \chi_1^{\log}(\mc O), \\
& \chi_2^{\log}(\mc O^S)= \chi_2^{\log}(\mc O)-\sum_{i\in S}\frac{l}{c_i-1}, \\
& (2g-2+b)l+\chi_1^{log}(\mc O)= \chi_2^{log}(\mc O),
 \end{sis}
\end{equation}
Moreover, using Clifford formula (which can be applied  since $\mc O$ is smoothable, one gets the upper bound (see \cite[Formula (14)]{CY}):
\begin{equation}\label{E:equa3}
     \chi_1^{log}(\mc O)\leq \frac{(g+1)l}{2}.
\end{equation}
Combining \eqref{E:chi1, chi2} and \eqref{E:equa3} and our assumption that $\chi_2^{\log}(\mc O^S)\leq 4\chi_1^{\log}(\mc O^S)$, we obtain
\begin{equation}\label{e:inequality}
g-7+2b - \sum_{i \in S} \frac{2}{c_i-1} \le 0.
\end{equation}

Using that $2g-2=\sum_i (c_i-2)$,  the inequality \eqref{e:inequality} is equivalent to:
\begin{equation}\label{E:new-ineq}
\sum_{i\in S} \left(\frac{c_i+2}{2}-\frac{2}{c_i-2} \right)+\sum_{i\in S^c} \frac{c_i+2}{2}\leq 6.    
\end{equation}
Using that $c_i\geq 2$ for any $i$ by Lemma \ref{L:ci}\eqref{L:ci2} and that $c_i \ge 3$ for any $i\in S$ by Lemma \ref{L:can-atom}, the inequality \eqref{E:new-ineq} implies that 
$$
\begin{sis}
&\frac{3}{2}|S|+2|S^c|\leq 6 \Rightarrow b\leq 4,\\
& c_i\leq 10 \text{ for any } i. \\
\end{sis}$$
Therefore, there are finitely many possibilities for $\un c$ and $S$ that we list below. We are going to use the fact that the singularities that we get are either of genus at most $3$ (which have been classified in \cite{Smyth1}, \cite{Battistella2}, and \cite{Battistella3}), or they belong to a non-varying strata as in \cite[Thm. 1.2]{CY} or they are unibranch of genus at most $5$ (and hence they correspond to a symmetric numerical semigroups of genus at most $5$, which are easy to list).

We end up with the following possibilities for $\un c$ and $S$:
\begin{enumerate}[leftmargin=*]    
\item $\un c=(3,3,3,3)$ (which implies $g=3$) and $S=\{1,2,3,4\}$.

Then  $\mc O$ is one of the curve singularities that correspond to the stratum $\mc H(1,1,1,1)$ in \cite[Table 1]{Battistella3}, and more explicitly, it is either a planar ordinary singularity of multiplicity $4$ of equation $\{xy(x-y)(x-ay)=0\}$ for $a\neq 0,1,\infty$, or the space curve singularity of equations $\{xy=z(z-x^2-y^2)=0\}$ (see \cite[5.11]{CY}).

In all cases, we have that (by \cite[3.7]{CY})
$$
\chi_1^{log}(\mc O^{\{1,2,3,4\}})=4 \text{ and }
    \chi_2^{log}(\mc O^{\{1,2,3,4\}})=\chi_2^{log}(\mc O)-4=20-4=16.
$$
    
    \item $\un c=(2,2,2)$ (which implies $g=1$) and $S=\emptyset$.

    Then $\mc O$ is a $D_4$-singularity and we have that (by  \cite[Table 1]{AFS16})
    $$\chi_1^{log}(\mc O)=1 \text{ and } \chi_2^{log}(\mc O)=4.$$
    \item $\un c=(3,3,2)$ (which implies $g=2$) and $S=\{1,2\}$ or $S=\{1\}$ or $S=\{2\}$.

   Then $\mc O$ is an $D_6$-singularity and we have (by \cite[Table 1]{AFS16}) 
   $$\chi_1^{log}(\mc O^{S})=3 \text{ and } \chi_2^{log}(\mc O^{S})=13-|S|.$$
    
    \item $\un c=(4,3,3)$ (which implies $g=3$) and $S=\{1,2,3\}$ or $S=\{2,3\}$.

Then $\mc O$ is the space curve singularity  of equations $\{xy,z^2-y^3+x^2z=0\}$ (which correspond to the stratum $\mc H(2,1,1)$ of \cite[Table 1]{Battistella3}, see also \cite[5.3]{CY}) and we have that (by \cite[5.3]{CY})
 $$\chi_1^{log}(\mc O^S)=11, \chi_2^{log}(\mc O^{2,3})=53-2\cdot6/2=47 \text{ and } \chi_2^{log}(\mc O^{1,2,3})=53-2\cdot6/2-6/3=45.$$
 Note that $\chi_2^{\log}(\mc O^S)>4 \chi_1^{\log}(\mc O^S)$ for both choices of $S$. 
    
   \item $\un c=(2,2)$ (which implies $g=1$) and $S=\emptyset$. 

   Then $\mc O$ is a tacnode  (i.e.\ an $A_3$-singularity) and we have that (by \cite[Table 1]{AFS16})
$$\chi_1^{log}(\mc O)=1 \text{ and } \chi_2^{log}(\mc O)=3.$$ 
   \item $\un c=(4,2)$ (which implies $g=2$) and $S=\{1\}$ or $S=\emptyset$. 

   Then  $\mc O$ is an $D_5$-singularity and we have (by \cite[Table 1]{AFS16}) 
   $$
   \chi_1^{log}(\mc O^S)=4 \text{ and } \chi_2^{log}(\mc O^S)=16-|S|. 
   $$
   
   \item $\un c=(3,3)$ (which implies $g=2$) and any $S$.

   Then $\mc O$ is an $A_5$-singularity and we have (by \cite[Table 1]{AFS16}) 
   $$ \chi_1^{log}(\mc O^S)=3 \text{ and } \chi_2^{\log}(\mc O^S)=11-|S|. $$
    \item $\un c=(6,2)$ (which implies $g=3$) and $S=\{1\}$ or $S=\emptyset$.

Then there are two possibilities:
\begin{itemize}
    \item $\mc O$ is a $D_7$-singularity (which corresponds to the stratum $\mc H(4,0)^{ev}$ of \cite[Table 1]{Battistella3}) and we have that (by \cite[Table 1]{AFS16}) 
$$
\chi_1^{log}(\mc O^S)=9 \text{ and }\chi_1^{log}(\mc O^S)=39-|S|. 
$$
    \item $\mc O$ is the singularity which corresponds to the stratum $\mc H(4,0)^{odd}$ of \cite[Table 1]{Battistella3} and we have that 
    $$
    \begin{sis}
     &    \chi_1^{log}(\mc O^S)=\chi_1^{log}(E_6)=8,\\
     & \chi_2^{log}(\mc O^S)=\chi_2^{log}(\mc O)-|S|=\chi_2^{\log}(E_6)+5-|S|=38-|S|,
    \end{sis}
$$
where we have used \cite[5.14]{CY} (with $l=5$) and the fact that $E_6$ is the singularity which corresponds to the stratum $\mc H(4)^{odd}$ by \cite[Table 1]{Battistella3}.
   \end{itemize}
  Note that $\chi_2^{\log}(\mc O^S)>4 \chi_1^{\log}(\mc O^S)$ in all cases.   
    
  \item $\un c=(5,3)$ (which implies $g=3$) and any $S$.
Then  $\mc O$ is an $E_7$-singularity and we have that (by \cite[Table 1]{AFS16}) 
$$\chi_1^{log}(\mc O^{\{S\}})=7, \chi_2^{log}(\mc O)=31, \chi_2^{log}(\mc O^{1})=30, \chi_2^{log}(\mc O^{2})=29, \chi_2^{log}(\mc O^{\{1,2\}})=28.$$
Note that $\chi_2^{log}(\mc O^S) > 4 \chi_1^{log}(\mc O^S)$ unless $S=\{1,2\}$.

  \item $\un c=(4,4)$ (which implies $g=3$) and any $S$.

 Then there are two possibilities:
 \begin{itemize}
     \item $\mc O$ is an $A_7$-singularity (which correspond to the stratum $\mc H(2,2)^{ev}$ of \cite[Table 1]{Battistella3}) and we have 
 (by \cite[Table 1]{AFS16})
 $$\chi_1^{log}(\mc O^S)=6 \text{ and } \chi_2^{log}(\mc O^S)=24-|S|.
 $$
 \item  $\mc O$ is the space curve singularity of equations $\{xy,x^3+y^3-z^2=0\}$ (which correspond to the stratum $\mc H(2,2)^{odd}$ of \cite[Table 1]{Battistella3}, see also \cite[5.2]{CY}) and we have that (by \cite[5.2]{CY})
 $$\chi_1^{log}(\mc O^S)=5 \text{ and } \chi_2^{log}(\mc O^S)=23-|S|.$$
 Note that $\chi_2^{\log}(\mc O^S)>4 \chi_1^{\log}(\mc O^S)$ for all choices of $S$.  
 \end{itemize}

    \item $\un c=(7,3)$ (which implies $g=4$) and $S=\{1,2\}$ or $S=\{2\}$.

    Then  $\mc O$ is the space curve singularity of equations $\{xz-y^2=x^2y-z^2=0\}$ (which correspond to the stratum $\mc H(5,1)$, see \cite[5.4]{CY}) and we have that (by \cite[5.4]{CY})
 $$\chi_1^{log}(\mc O^S)=12, \chi_2^{log}(\mc O^{2})=60-6/2=57 \text{ and } \chi_2^{log}(\mc O^{1,2})=60-6/6-6/2=56.$$
 Note that $\chi_2^{\log}(\mc O^S)>4 \chi_1^{\log}(\mc O^S)$ for both choices of $S$. 
    
    \item $\un c=(6,4)$ (which implies $g=4$) and $S=\{1,2\}$.
Then there are three possibilities:
\begin{itemize}
    \item $\mc O$  is the space curve singularity of equations $\{xy=x^5+y^3-z^2=0\}$ (which corresponds to the stratum $\mc H(4,2)^{ev}$ of \cite[5.5]{CY}) and we have that (by \cite[5.5]{CY})
    $$ \chi_1(\mc O^{\{1,2\}})=32 \text{ and } \chi_2(\mc O^{\{1,2\}})=152-15/5-15/33=144.
    $$
    \item  $\mc O$ is the  curve singularity in $\bb A^4$ having equations $\{xy=xw=z-w^2=zw-y^3=x^3+y^2w-z^2=0\}$ (which corresponds to the stratum $\mc H(4,2)^{odd}$ of \cite[5.5]{CY}) and we have that (by \cite[5.5]{CY})
    $$ \chi_1(\mc O^{\{1,2\}})=29 \text{ and } \chi_2(\mc O^{\{1,2\}})=149-15/5-15/33=141.
    $$
    \item  $\mc O$ is the  curve singularity which corresponds to the stratum $\mc H(4,2)^{hyp}$ (see \cite[5.15]{CY}) and we have that (by \cite[6.2]{CY} and \cite[Prop. 4.1(iv)]{CY})
    $$
    \begin{sis}
     &    \chi_1(\mc O^{\{1,2\}})=\chi_1(\mc O)=\frac{(g+1)l}{2}-\frac{l-a_1}{4}-\frac{l-a_2}{4}=32, \\
     & \chi_2(\mc O^{\{1,2\}})=\chi_2(\mc O)-15/5-15/3=\chi_1(\mc O)+(2g-2+b)l-3-5=144.
    \end{sis}
    $$
   
\end{itemize}
In all of the above cases, we have that  $\chi_2^{log}(\mc O^{\{1,2\}}) > 4 \chi_1^{log}(\mc O^{\{1,2\}})$. 
    
    \item $\un c=(5,5)$ (which implies $g=4$) and $S=\{1,2\}$.

Then there are two possibilities:
 \begin{itemize}
     \item $\mc O$ is an $A_9$-singularity (which correspond to the stratum $\mc H(3,3)^{hyp}$ of \cite[3.2]{CY}) and we have 
 (by \cite[Table 1]{AFS16})
 $$\chi_1^{log}(\mc O^{\{1,2\}})=10 \text{ and } \chi_2^{log}(\mc O^{\{1,2\}})=42-1-1=40.
 $$
 \item $\mc O$  is the space curve singularity of equations $\{yz=x^3-y^2-z^2=0\}$ (which corresponds to the stratum $\mc H(3,3)^{non-hyp}$ of \cite[5.6]{CY}) and we have (by \cite[5.6]{CY})
 $$\chi_1^{log}(\mc O^{\{1,2\}})=8 \text{ and } \chi_2^{log}(\mc O^{\{1,2\}})=40-1-1=38.$$
 \end{itemize}
 Note that in the second case we have that   $\chi_2^{log}(\mc O^{\{1,2\}}) > 4 \chi_1^{log}(\mc O^{\{1,2\}})$. 
  
    \item $\un{c}=(2)$ (which implies $g=1$) and $S=\emptyset$. 
    
    Then $\mc O$ is an ordinary cusp (i.e.\ an $A_2$-singularity) and we have that (by \cite[Table 1]{AFS16}) 
    $$ \chi_1^{log}(\mc O)=1 \text{ and } \chi_2^{log}(\mc O)=2.$$ 
    \item $\un{c}=(4)$ (which implies $g=2$) and $S=\{1\}$ or $S=\emptyset$.

    Then $\mc O$ is a ramphoid cusp (i.e. an $A_4$-singularity) and we have that (by \cite[Table 1]{AFS16})
 $$
   \chi_1^{log}(\mc O^S)=4,    \chi_2^{log}(\mc O)=13 \text{ and }  \chi_2^{log}(\mc O^{1})=13-3/3=12.
   $$
    \item $\un{c}=(6)$ (which implies $g=3$) and $S=\{1\}$ or $S=\emptyset$.

    Then $\mc O$ is the monomial curve singularity associated to one of the following symmetric numerical semigroups:
    \begin{itemize}
        \item $\langle 2,7\rangle$ (which corresponds to the $A_{6}$-singularity) which has gap sequence equal to $\{1,3,5\}$. We have (by \cite[Table 1]{AFS16})   
        $$
        \chi_1^{\log}(\mc O)=\chi_1^{\log}(\mc O^{1})=9, \chi_2^{\log}(\mc O)=25+9=34 \text{ and } \chi_2^{\log}(\mc O^{1})=34-5/5=33.
        $$ 
        \item $\langle 3,4\rangle$ (which corresponds to the $E_{6}$-singularity) which has gap sequence equal to $\{1,2,5\}$. We have (by \cite[Table 1]{AFS16})   
        $$
        \chi_1^{\log}(\mc O)=\chi_1^{\log}(\mc O^{1})=8, \chi_2^{\log}(\mc O)=25+8=33 \text{ and } \chi_2^{\log}(\mc O^{1})=33-5/5=32.
        $$ 
        Note that $\chi_2^{\log}(\mc O)>4 \chi_1^{\log}(\mc O)$. 
    \end{itemize}
     \item $\un{c}=(8)$ (which implies $g=4$) and $S=\{1\}$ or $S=\emptyset$.

     Then $\mc O$ is the monomial curve singularity associated to one of the following symmetric numerical semigroups:
    \begin{itemize}
        \item $\langle 2,9\rangle$ (which corresponds to the $A_{8}$-singularity) which has gap sequence equal to $\{1,3,5,7\}$.  We have (by \cite[Table 1]{AFS16})   
        $$
        \chi_1^{\log}(\mc O)=\chi_1^{\log}(\mc O^{1})=16, \chi_2^{\log}(\mc O)=49+16=65 \text{ and } \chi_2^{\log}(\mc O^{1})=65-7/7=64.
        $$  
        Note that $\chi_2^{\log}(\mc O)>4 \chi_1^{\log}(\mc O)$.
        \item $\langle 3,5\rangle$ (which corresponds to the $E_{8}$-singularity) which has gap sequence equal to $\{1,2,4,7\}$. We have (by \cite[Table 1]{AFS16})   
        $$
        \chi_1^{\log}(\mc O)=\chi_1^{\log}(\mc O^{1})=14, \chi_2^{\log}(\mc O)=49+14=63 \text{ and } \chi_2^{\log}(\mc O^{1})=63-7/7=62.
        $$  
         Note that $\chi_2^{\log}(\mc O)>4 \chi_1^{\log}(\mc O)$ and  $\chi_2^{\log}(\mc O^{1})>4 \chi_1^{\log}(\mc O^{1})$.
        \item $\langle 4,5,6\rangle$  which has gap sequence equal to $\{1,2,3,7\}$. We have (by \cite[Table 1]{AFS16})   
        $$
        \chi_1^{\log}(\mc O)=\chi_1^{\log}(\mc O^{1})=13, \chi_2^{\log}(\mc O)=49+13=62 \text{ and } \chi_2^{\log}(\mc O^{1})=62-7/7=61.
        $$  
         Note that $\chi_2^{\log}(\mc O)>4 \chi_1^{\log}(\mc O)$ and  $\chi_2^{\log}(\mc O^{1})>4 \chi_1^{\log}(\mc O^{1})$.
    \end{itemize}
     \item $\un{c}=(10)$ (which implies $g=5$) and $S=\{1\}$ or $S=\emptyset$.

Then $\mc O$ is the monomial curve singularity associated to one of the following symmetric numerical semigroups:
    \begin{itemize}
        \item $\langle 2,11\rangle$ (which corresponds to the $A_{10}$-singularity) which has gap sequence equal to $\{1,3,5,7,9\}$. We have (by \cite[Table 1]{AFS16})   
        $$
        \chi_1^{\log}(\mc O)=\chi_1^{\log}(\mc O^{1})=25, \chi_2^{\log}(\mc O)=81+25=106 \text{ and } \chi_2^{\log}(\mc O^{1})=106-9/9=105.
        $$  
         Note that $\chi_2^{\log}(\mc O)>4 \chi_1^{\log}(\mc O)$ and  $\chi_2^{\log}(\mc O^{1})>4 \chi_1^{\log}(\mc O^{1})$. 
        \item $\langle 3,7\rangle$  which has gap sequence equal to $\{1,2,4,5,9\}$.  We have (by \cite[Table 1]{AFS16})   
        $$
        \chi_1^{\log}(\mc O)=\chi_1^{\log}(\mc O^{1})=21, \chi_2^{\log}(\mc O)=81+21=102 \text{ and } \chi_2^{\log}(\mc O^{1})=102-9/9=101.
        $$  
         Note that $\chi_2^{\log}(\mc O)>4 \chi_1^{\log}(\mc O)$ and  $\chi_2^{\log}(\mc O^{1})>4 \chi_1^{\log}(\mc O^{1})$. 
        \item $\langle 4,5,11\rangle$  which has gap sequence equal to $\{1,2,3,6,9\}$.  We have (by \cite[Table 1]{AFS16})   
        $$
        \chi_1^{\log}(\mc O)=\chi_1^{\log}(\mc O^{1})=21, \chi_2^{\log}(\mc O)=81+21=102 \text{ and } \chi_2^{\log}(\mc O^{1})=102-9/9=101.
        $$  
         Note that $\chi_2^{\log}(\mc O)>4 \chi_1^{\log}(\mc O)$ and  $\chi_2^{\log}(\mc O^{1})>4 \chi_1^{\log}(\mc O^{1})$. 
         \item $\langle 4,6,7\rangle$  which has gap sequence equal to $\{1,2,3,5,9\}$.  We have (by \cite[Table 1]{AFS16})   
        $$
        \chi_1^{\log}(\mc O)=\chi_1^{\log}(\mc O^{1})=20, \chi_2^{\log}(\mc O)=81+20=101 \text{ and } \chi_2^{\log}(\mc O^{1})=101-9/9=100.
        $$  
         Note that $\chi_2^{\log}(\mc O)>4 \chi_1^{\log}(\mc O)$ and  $\chi_2^{\log}(\mc O^{1})>4 \chi_1^{\log}(\mc O^{1})$. 
         \item $\langle 5,6\rangle$  which has gap sequence equal to $\{1,2,3,4,9\}$.  We have (by \cite[Table 1]{AFS16})   
        $$
        \chi_1^{\log}(\mc O)=\chi_1^{\log}(\mc O^{1})=19, \chi_2^{\log}(\mc O)=81+19=100 \text{ and } \chi_2^{\log}(\mc O^{1})=100-9/9=99.
        $$  
         Note that $\chi_2^{\log}(\mc O)>4 \chi_1^{\log}(\mc O)$ and  $\chi_2^{\log}(\mc O^{1})>4 \chi_1^{\log}(\mc O^{1})$. 
    \end{itemize}
    
\end{enumerate}

\end{proof}

\begin{theorem}\label{T:destab} 
Fix $5/9 \le \alpha \le 1$. If $(C,p_i)\in \mc U_{g,n}(\alpha)$, then the singularities of $(C,p_i)$ are at worst those listed in Table \ref{t:anyS} with $\alpha$-invariant greater than or equal to $\alpha$ and in such a way that  the dangling branches have arithmetic genus $0$ and no marked points. 
\end{theorem}
\begin{proof}
  Pick a non-nodal (Gorenstein) point $q\in C$ and consider the morphism $T:\Theta\to \mc U_{g,n}$ with $T(1)=(C,p_i)$ associated to the test configuration $f:(\mc C,\sigma_i)\to \bb A^1$ for $C$ constructed in Theorem \ref{T:deg-sing}. 
  Corollary \ref{C:deg-sing} and the assumption that $(C,p_i)\in \mc U_{g,n}(\alpha)$ imply that 
  $$\wei_T(K+\psi+\alpha(\delta-\psi))=w_{\alpha}(\mc O^S)\geq 0.$$
Then, we apply Remark \ref{R:alpha-inv}\eqref{R:alpha-inv2} in order to get that 
$$
\frac{5}{9}\leq \alpha \leq \alpha(\mc O^S)<1.
$$
Using that $(C,p_i)$ is smoothable, which implies that $\ov X(\mc O)^S$ is smoothable, we are in a position of applying Lemma \ref{l:4}: we deduce that $\mc O^S$ is one of the singularities listed in Table \ref{t:anyS} and with $\alpha(\mc O^S)\geq \alpha$. 
We now conclude using that the singularities listed in Table \ref{t:anyS} have no non-trivial isotrivial deformations (and hence that $\wh{\mc O}_{C,q}\cong \mc O$) and the discussion following Theorem \ref{T:deg-sing}. 
\end{proof}

\begin{corollary}\label{C:destab9/11}
  If $(C,p_i)\in \mc U_{g,n}$ is smoothable and non-nodal then 
  $(C,p_i)$ is not $\alpha$-semistable for every $\alpha> 9/11$.
\end{corollary}

In the last part of this Section, we want to compare the stacks $\mc U_{g,n}(\alpha)$ for $\alpha>2/3-\epsilon$ with the first steps of the Hassett-Keel program carried over by  Alper-Fedorchuck-Smyth-van der Wyck \cite{AFSV1, AFS2, AFS3}.  

Firs of all, recall the definitions of the stacks 
$\ov{\mc M}_{g,n}(\alpha)$ for $\alpha\in (2/3 -\varepsilon,1]$ introduced in \cite{AFSV1}, referring to loc. cit. for the terminology. 

\begin{definition}[{\cite[Definition 2.5]{AFSV1}}]\label{D:Mgalpha}
For each $\alpha\in (2/3 -\varepsilon,1]$, where $0<\varepsilon \ll 1$, define $\ov{\mc M}_{g,n}(\alpha)$ to be the substack of
$\mathcal U_{g,n}$ whose objects are pointed curves $(C,p_i)$ satisfying the
following conditions, according to the value of~$\alpha$:
\begin{enumerate}
    \item if $\alpha\in(9/11,1]$, then $C$ has only $A_1$-singularities;
    \item if $\alpha=9/11$, then $C$ has only $A_1, A_2$-singularities;
    \item if $\alpha\in(7/10,9/11)$, then $(C,p_i)\in\ov{\mc M}_{g,n}(9/11)$ and does not contain
    \begin{itemize}
        \item $A_1$-attached elliptic tails;
    \end{itemize}
    \item if $\alpha=7/10$, then $C$ has only $A_1, A_2,A_3$-singularities and does not contain
    \begin{itemize}
        \item $A_1,A_3$-attached elliptic tails;
    \end{itemize}
     \item if $\alpha \in (2/3, 7/10)$, then $(C,p_i)\in\ov{\mc M}_{g,n}(7/10)$  and does not contain
    \begin{itemize}
        \item $A_1/A_1$-attached elliptic chains;
    \end{itemize}
     \item if $\alpha=2/3$, then $C$ has only $A_1, A_2,A_3,A_4$-singularities and does not contain
    \begin{itemize}
        \item $A_1,A_3,A_4$-attached elliptic tails;
        \item $A_1/A_1, A_1/A_4, A_4/A_4$-attached elliptic chains;
    \end{itemize}
    \item if $\alpha \in (2/3-\varepsilon, 2/3)$, then $(C,p_i)\in\ov{\mc M}_{g,n}(2/3)$  and does not contain
    \begin{itemize}
        \item $A_1$-attached Weierstrass chains.
    \end{itemize}
\end{enumerate}
\end{definition}

The following result compares our stacks $\mc U_{g,n}(\alpha)$ with the stacks $\ov{\mc M}_{g,n}(\alpha)$ of \cite{AFSV1}, for $\alpha \in  (2/3-\varepsilon,1]$. 

\begin{theorem}\label{thm:U(alpha)}
Let  $\alpha \in  (2/3-\varepsilon,1]$, where $0 < \varepsilon \ll 1$. Then
$$
\mc U_{g,n}(\alpha) \subset \ov{\mc M}_{g,n}(\alpha).
$$
\end{theorem}
\begin{proof}
Let $(C,p_i) \in \mc U_{g,n}(\alpha)$, i.e. $(C,p_i)$ is smoothable and $\alpha$-semistable in the sense of Definition \ref{D:Ugalpha}.
Theorem \ref{T:destab} shows that, for the given value of $\alpha$, the curve $C$ satisfies the singularity requirements defining $\ov{\mc M}_{g,n}(\alpha)$. Thus, we need to show that $C$ does not contain any of the additional configurations excluded in the Definition \ref{D:Mgalpha} of $\ov{\mc M}_{g,n}(\alpha)$. 

We are going to use the following three special cases of test configurations constructed in Theorem \ref{T:deg-CY}:
\begin{enumerate}[(A)]
    \item If $(\wh{\mc C},\wh{\sigma}_i)$ is a test configuration whose central fiber $C$ has a $\bb G_m$-invariant \emph{$A_1$-attached elliptic tail} $(E,q)$, then taking $\omega=0$ we get a new test configuration $(\mc C,\sigma_i)$ for $(\wh{\mc C}_1,\wh{\sigma}_i(1))$ whose special fiber is obtained by gluing the complementary subcurve $E^c$ with $\ov X(^{\opp}A_2)$ at $q$.
    Corollary \ref{C:deg-CY} and Table \ref{t:anyS} imply that the contribution of the newly created $\ov X(^{\opp}A_2)$-atom to the weight of $K+\psi+\alpha(\delta-\psi)$ is equal to 
    $$
    -w_{\alpha}(A_2)=(2-\alpha)\chi_2^{\log}(A_2)+(13\alpha-13)\chi_1^{\log}(A_2)=(2-\alpha)\cdot 2+(13\alpha-13)\cdot 1=11\alpha-9.
    $$
    \item If $(\wh{\mc C},\wh{\sigma}_i)$ is a test configuration whose central fiber $C$ has a $\bb G_m$-invariant \emph{$A_1/A_1$-attached elliptic bridge} $(E,q_1,q_2)$, then taking $\omega=0$ we get a new test configuration $(\mc C,\sigma_i)$ for $(\wh{\mc C}_1,\wh{\sigma}_i(1))$ whose special fiber is obtained by gluing the complementary subcurve $E^c$ with $\ov X(^{\opp}A_3)$ at $q_1$ and $q_2$. Corollary \ref{C:deg-CY} and Table \ref{t:anyS} imply that the contribution of the newly created $\ov X(^{\opp}A_3)$-atom to the weight of $K+\psi+\alpha(\delta-\psi)$ is equal to 
    $$
    -w_{\alpha}(A_3)=(2-\alpha)\chi_2^{\log}(A_3)+(13\alpha-13)\chi_1^{\log}(A_3)=(2-\alpha)\cdot3+(13\alpha-13)\cdot1=10\alpha-7.
    $$
    \item If $(\wh{\mc C},\wh{\sigma}_i)$ is a test configuration whose central fiber $C$ has a $\bb G_m$-invariant \emph{$A_1$-attached Weierstrass tail} $(D,q)$, i.e. $D$ is a genus $2$ curve and $q$ is a Weierstrass point, then taking $\omega$ such that $\div(\omega)=2q$ we get a new test configuration $(\mc C,\sigma_i)$ for $(\wh{\mc C}_1,\wh{\sigma}_i(1))$ whose special fiber is obtained by gluing the complementary subcurve $D^c$ with $\ov X(^{\opp}A_4)$ at $q$. Corollary \ref{C:deg-CY} and Table \ref{t:anyS} imply that the contribution of the newly created $\ov X(^{\opp}A_4)$-atom to the weight of $K+\psi+\alpha(\delta-\psi)$ is equal to
    $$
    -w_{\alpha}(A_4)=(2-\alpha)\chi_2^{\log}(A_4)+(13\alpha-13)\chi_1^{\log}(A_4)=(2-\alpha)\cdot13+(13\alpha-13)\cdot4=39\alpha-26.
    $$
\end{enumerate}

If $\alpha \ge 9/11$, there is nothing left to prove. This covers cases (1) and (2) in Definition \ref{D:Mgalpha}, so we start with Item (3).

\begin{enumerate}[leftmargin=*]\setcounter{enumi}{2}  
    \item Assume that $\alpha<9/11$ and suppose, by contradiction, that $(C,p_i)$
contains an $A_1$-attached elliptic tail $(E,q)$. We can apply construction (A) to $(E,q)$ in order to produce a test configuration $T$ for $(C,p_i)$ such that 
$$
\wei_T(K+\psi+\alpha(\delta-\psi))=11\alpha-9<0,
$$
which contradicts the $\alpha$-semistability of $(C,p_i)$.

\item Assume that $\alpha \le 7/10$ and suppose, by contradiction, that
$(C,p_i)$ contains an elliptic tail $E$ attached at an $A_3$-singularity
$q\in C$. We construct a test configuration $T$ as follows. First, applying
Theorem \ref{T:deg-sing} to the singularity $q$, we replace it with an $\ov X(A_3)$-atom. The resulting curve still contains an $A_1$-attached elliptic tail, which we then replace by a $\ov X(^{\opp}A_2)$-atom as in the construction (A), 
see Figure \ref{F:A3elliptictail}.

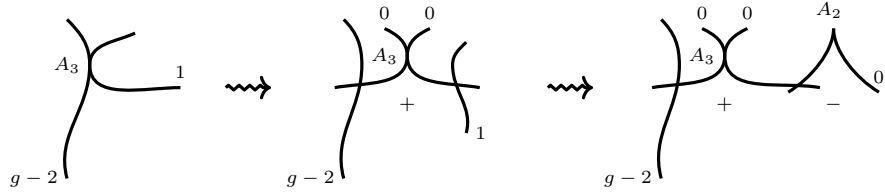
\begin{figure}[!h]
	\centering
	
	\begin{tikzpicture}[scale=0.6, every node/.style={font=\scriptsize}]

		\coordinate (x1) at (0-1, 1);
		\coordinate (y1) at (0.5-1, 0);
		\coordinate (z1) at (0-1, -2.5);
		\node[left] at (z1) {$g-2$};
		\draw [very thick] (x1) to[in=90, out=-45] (y1) to[in=105, out=-90] (z1);
		
		\coordinate (a1) at (1.5-1, 0.7);
		\coordinate (b1) at (0.5-1, 0);
		\coordinate (c1) at (2.5-1, -0.5);
		\node[above] at (c1) {1};
		\draw [very thick] (a1) to[out=205, in=90] (b1) to[out=-90, in=180] (c1);
		
		\node[left] at (y1) {$A_3$};
		
		
		\draw[very thick, ->,
		line join=round,
		decorate, decoration={
			zigzag,
			segment length=4,
			amplitude=.9, post=lineto,
			post length=2pt
		}] (2.5, -0.5) -- (3.5, -0.5);
		
		
		\coordinate (b2) at (6.5, 0.2);       
		\coordinate (c2) at (6.0, 0.8);       
		\coordinate (d2) at (7.0, 0.8);       
		\node[above] at (c2) {0};
		\node[above] at (d2) {0};
		
		\draw [very thick] (4.9, -0.5) to[out=10, in=-90] (b2) to[out=90, in=-25] (c2);
		
		\draw [very thick] (d2) to[out=205, in=90] (b2) to[out=-90, in=170] (8.1, -0.5);
		
		\coordinate (x3) at (5.1, 1);
		\coordinate (y3) at (5.5, 0); 
		\coordinate (z3) at (5.1, -2.5);
		\node[left] at (z3) {$g-2$};
		\draw [very thick] (x3) to[in=90, out=-45] (y3) to[in=105, out=-90] (z3);
		
		\coordinate (x4) at (7.8, 0.5);       
		\coordinate (y4) at (7.5, 0);         
		\coordinate (z4) at (7.8, -1.5);      
		\node[right] at (z4) {1};
		\draw [very thick] (x4) to[in=90, out=-135] (y4) to[in=75, out=-90] (z4);
		
		\node[below] at (6.5, -0.5) {$+$};
		\node[left] at (6.5, 0.2) {$A_3$};
		
		\draw[very thick, ->,
		line join=round,
		decorate, decoration={
			zigzag,
			segment length=4,
			amplitude=.9, post=lineto,
			post length=2pt
		}] (9.6, -0.5) -- (10.6, -0.5);
		
	
		\coordinate (b3) at (13.5, 0.2);      
		\coordinate (c3) at (13.0, 0.8);      
		\coordinate (d3) at (14.0, 0.8);      
		\node[above] at (c3) {0};
		\node[above] at (d3) {0};
		\node[left] at (13.5, 0.2) {$A_3$};
		\draw [very thick] (11.9, -0.5) to[out=10, in=-90] (b3) to[out=90, in=-25] (c3);
		
		\draw [very thick] (d3) to[out=205, in=90] (b3) to[out=-90, in=170] (15.6, -0.5);
		
		\coordinate (x5) at (12.1, 1);
		\coordinate (y5) at (12.5, 0); 
		\coordinate (z5) at (12.1, -2.5);
		\node[left] at (z5) {$g-2$};
		\draw [very thick] (x5) to[in=90, out=-45] (y5) to[in=105, out=-90] (z5);
		
		\coordinate (cusp_start) at (14.9, -0.6); 
		\coordinate (cusp_tip) at (15.9, 0.8);    
		\coordinate (cusp_end) at (16.9, -0.6);
		\node[above] at (cusp_end) {0};
		
		\draw [very thick] (cusp_start) .. controls (15.4, -0.3) and (15.9, 0.3) .. (cusp_tip);
		\draw [very thick] (cusp_end) .. controls (16.4, -0.3) and (15.9, 0.3) .. (cusp_tip);
		
		\node[below] at (13.5, -0.5) {$+$};
		\node[below] at (15.9, -0.5) {$-$};
		\node[left] at (16.3, 1.2) {$A_2$};
	\end{tikzpicture}
	\caption{A genus $g$ curve with an $A_3$-attached elliptic tail. First the $A_3$-singularity is replaced by an $\ov X(A_3)$-atom and then the $A_1$-attached elliptic tails is replaced by an $\ov X(^{\opp}A_2)$-atom.}\label{F:A3elliptictail}
\end{figure}

We now compute the weight at $T$ as the sum of the contributions of the two atoms $\ov X(A_3)$ and $\ov X(^{\opp}A_2)$ in the central fiber:
$$\wei_T(K+\psi+\alpha(\delta-\psi))=w_{\alpha}(A_3)-w_{\alpha}(A_2)=(7-10\alpha)-(9-11\alpha)=\alpha-2<0,$$
contradicting the $\alpha$-semistability of $(C,p_i)$.

\item Assume that $\alpha<7/10$ and suppose, by contradiction, that
$(C,p_i)$ contains an $A_1/A_1$-attached elliptic chain $E_1\cup\cdots\cup E_r$ of length $r\geq 1$.  We construct a test configuration
$T$ as follows. First, for every
$A_3$-singularity joining two consecutive elliptic components, we apply
Theorem \ref{T:deg-sing} to replace it with an $\ov X(A_3)$-atom. The resulting
curve consists of a chain of $A_1/A_1$-attached elliptic bridges joined by
$\ov X(A_3)$-atoms. We then replace $A_1/A_1$-attached elliptic bridge  by a $\ov X(^{\opp}A_3)$-atom as in the construction (B), see Figure \ref{F:ellipticbridge}.

 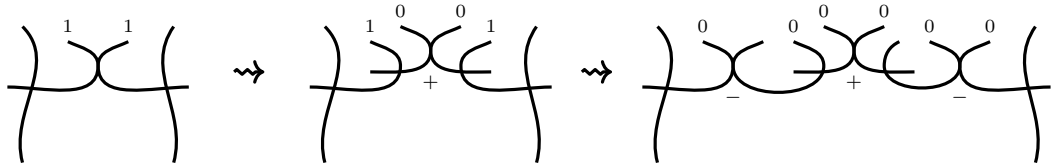
\begin{figure}[!h]
	\centering
	\begin{tikzpicture}[scale=0.4, every node/.style={font=\scriptsize}]

\coordinate (x) at (-0.5,2);
\coordinate (y) at (-0.5, -2.5);
\draw [very thick, in=105, out=-45] (x) to (y);

\coordinate (a) at (-1, 0);
\coordinate (b) at (2, 0.7);
\coordinate[label=above:1] (c) at (1, 1.5);
\draw [very thick] (a)   to[out=0, in=-90]  (b)  to[out=90, in=-25] (c);

\coordinate[label=above:1] (d) at (0+3, 1.5);
\coordinate (e) at (1+4, 0);
\draw [very thick] (d)  to[out=205, in=90]   (b)  to[out=-90, in=180] (e);

\coordinate (w) at (4.5,2);
\coordinate (z) at (4.5, -2.5);
\draw [very thick, in=75, out=-125] (w) to (z);

	\draw[very thick, ->,
line join=round,
decorate, decoration={
	zigzag,
	segment length=4,
	amplitude=.9, post=lineto,
	post length=2pt
}] (6.5, 0.5) -- (7.5, 0.5);

\coordinate (x) at (-0.5+10,2);
\coordinate (y) at (-0.5+10, -2.5);
\draw [very thick, in=105, out=-45] (x) to (y);

\coordinate (a) at (-1+10, 0);
\coordinate (b) at (2+10, 0.7);
\coordinate[label=above:1] (c) at (1+10, 1.5);
\draw [very thick] (a)   to[out=0, in=-90]  (b)  to[out=90, in=-25] (c);

\coordinate (A) at (-1+12, 0.5);
\coordinate (B) at (2+11, 1.2);
\coordinate[label=above:0] (C) at (1+11, 2);
\draw [very thick] (A)   to[out=0, in=-90]  (B)  to[out=90, in=-25] (C);
\node[below] at (2+11, 0.7) {$+$};

\coordinate[label=above:0] (D) at (0+3+11, 2);
\coordinate (E) at (1+4+10, 0.5);
\draw [very thick] (D)  to[out=205, in=90]   (B)  to[out=-90, in=180] (E);

\coordinate[label=above:1] (d) at (0+3+12, 1.5);
\coordinate (e) at (1+4+12, 0);
\draw [very thick] (d)  to[out=205, in=90]   (2+12, 0.7)  to[out=-90, in=180] (e);

\coordinate (w) at (4.5+12,2);
\coordinate (z) at (4.5+12, -2.5);
\draw [very thick, in=75, out=-125] (w) to (z);

\draw[very thick, ->,
line join=round,
decorate, decoration={
	zigzag,
	segment length=4,
	amplitude=.9, post=lineto,
	post length=2pt
}] (18, 0.5) -- (19, 0.5);


\coordinate (x) at (-0.5+21,2);
\coordinate (y) at (-0.5+21, -2.5);
\draw [very thick, in=105, out=-45] (x) to (y);

\coordinate (a) at (-1+21, 0);
\coordinate (b) at (2+21, 0.7);
\coordinate[label=above:0] (c) at (1+21, 1.5);
\draw [very thick] (a) to[out=0, in=-90] (b) to[out=90, in=-25] (c);
\node[below] at (2+21, 0.2) {$-$};
\coordinate (d) at (3+21, 1.5);
\coordinate (e) at (5+21, 0.7);
\coordinate[label=above:0] (f) at (4+21, 1.5); 
\draw [very thick] (d) to[out=205, in=90] (b) to[out=-90, in=-90] (e) to[out=90, in=-25] (f);

\coordinate (A) at (-1+12+14, 0.5);
\coordinate (B) at (2+11+14, 1.2);
\coordinate[label=above:0] (C) at (1++11+14, 2);
\draw [very thick] (A)   to[out=0, in=-90]  (B)  to[out=90, in=-25] (C);
\node[below] at (2+11+14, 0.7) {$+$};

\coordinate[label=above:0] (D) at (0+3+11+14, 2);
\coordinate (E) at (1+4+10+14, 0.5);
\draw [very thick] (D)  to[out=205, in=90]   (B)  to[out=-90, in=180] (E);

\coordinate (d) at (3+21+4.5, 1.5);
\coordinate (e) at (5+21+4.5, 0.7);
\coordinate[label=above:0] (f) at (4+21+4.5, 1.5); 
\draw [very thick] (d) to[out=205, in=90] (2+21+5, 0.7) to[out=-90, in=-90] (e) to[out=90, in=-25] (f);
\node[below] at (4+21+5.5, 0.2) {$-$};
\coordinate[label=above:0] (l) at (6+21+4.5, 1.5);
\coordinate (m) at (8+21+4.5, 0);
\draw [very thick] (l) to[out=205, in=90] (5+21+4.5, 0.7) to[out=-90, in=180] (m);

\coordinate (w) at (7.5+21+4.5,2);
\coordinate (z) at (7.5+21+4.5, -2.5);
\draw [very thick, in=75, out=-125] (w) to (z);

\end{tikzpicture}
	\caption{A curve with an $A_1/A_1$-attached elliptic chain of length $r=2$. First the $A_3$-singularity is replaced by an $\ov X(A_3)$-atom and then the $A_1/A_1$-attached elliptic bridges are replaced by  $\ov X(^{\opp}A_3)$-atoms.}\label{F:ellipticbridge}
\end{figure}

We now compute the weight at $T$ as the sum of the contributions of the $(r-1)$ atoms $\ov X(A_3)$ and of the $r$ atoms $\ov X(^{\opp}A_3)$ in the central fiber:
$$
\wei_T(K+\psi+\alpha(\delta-\psi))=(r-1)w_{\alpha}(A_3)-rw_{\alpha}(A_3)=-w_{\alpha}(A_3)=10\alpha-7<0, 
$$
which contradicts the $\alpha$-semistability of $(C,p_i)$.

\item Assume $\alpha \in (2/3-\varepsilon,2/3]$. 
Suppose, by contradiction, that $(C,p_i)$ contains an $A_4$-attached elliptic tail. Then $C$ itself is equal to the the elliptic tail, i.e. $C$ is a genus $g=3$ curve with an $A_4$-singularity and $n=0$. We first apply Theorem \ref{T:deg-sing} and we obtain a $\ov X(A_4)$-atom and an $A_1$-attached elliptic tail. Then, we apply construction (A) in order to replace the $A_1$-attached elliptic tail with a $\ov X(^{\opp}A_2)$-atom. We now compute the weight at $T$ as the sum of the contributions of the two atoms $\ov X(A_4)$ and $\ov X(^{\opp}A_2)$ in the central fiber:
$$\wei_T(K+\psi+\alpha(\delta-\psi))=w_{\alpha}(A_4)-w_{\alpha}(A_2)= (26-39\alpha)-(9-11 \alpha)=17-28\alpha<0,
$$
since $\alpha >17/28$, which contradicts the $\alpha$-semistability of $(C,p_i)$.

Suppose, by contradiction, that $(C,p_i)$ contains an $A_1/A_4$-attached elliptic chain $E=E_1\cup\cdots\cup E_r$, with $r$ bridges and
$r-1$ internal tacnodes, one end a node and the other an $A_4$. We first apply Theorem \ref{T:deg-sing} to replace each of the the $r-1$ internal tacnodes by a $\ov X(A_3)$-atom and the $A_4$ end by a $\ov X(A_4)$-atom. The resulting
curve consists of a chain of $A_1/A_1$-attached elliptic bridges joined by
$\ov X(A_3)$-atoms and with a $\ov X(A_4)$-atom at the end.  We then replace each $A_1/A_1$-attached elliptic bridge  by a $\ov X(^{\opp}A_3)$-atom as in the construction (B), see  Figure \ref{F:A4ellipticbridge}.

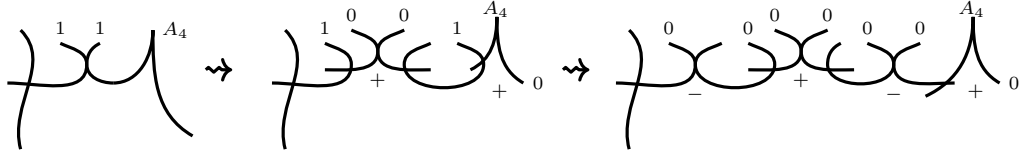
\begin{figure}[!h]
	\centering
	\begin{tikzpicture}[scale=0.35, every node/.style={font=\scriptsize}]
		
		\coordinate (x) at (-0.5,2);
		\coordinate (y) at (-0.5, -2.5);
		\draw [very thick, in=105, out=-45] (x) to (y);
		
		\coordinate (a) at (-1, 0);
		\coordinate (b) at (2, 0.7);
		\coordinate[label=above:1] (c) at (1, 1.5);
		\draw [very thick] (a)   to[out=0, in=-90]  (b)  to[out=90, in=-25] (c);

		\coordinate[label=above:1] (d) at (0+2.5, 1.5);
		\coordinate (e) at (1+2, 0);
		\draw [very thick] (d)  to[out=205, in=90]   (b)  to[out=-90, in=180] (e);
		
		\coordinate (x) at (-1+4,0);
		\coordinate[label=right:$A_4$] (y) at (0+4.5, 2);
		\coordinate (z) at (1+5	, -2);
		\draw [very thick] (x) to[in=-90, out=0] (y) to[in=150, out=-90] (z);

		\draw[very thick, ->,
		line join=round,
		decorate, decoration={
			zigzag,
			segment length=4,
			amplitude=.9, post=lineto,
			post length=2pt
		}] (6.5, 0.5) -- (7.5, 0.5);

		\coordinate (x) at (-0.5+10,2);
		\coordinate (y) at (-0.5+10, -2.5);
		\draw [very thick, in=105, out=-45] (x) to (y);
		
		\coordinate (a) at (-1+10, 0);
		\coordinate (b) at (2+10, 0.7);
		\coordinate[label=above:1] (c) at (1+10, 1.5);
		\draw [very thick] (a)   to[out=0, in=-90]  (b)  to[out=90, in=-25] (c);
		
		\coordinate (A) at (-1+12, 0.5);
		\coordinate (B) at (2+11, 1.2);
		\coordinate[label=above:0] (C) at (1+11, 2);
		\draw [very thick] (A)   to[out=0, in=-90]  (B)  to[out=90, in=-25] (C);
		\node[below] at (2+11, 0.7) {$+$};
		
		\coordinate[label=above:0] (D) at (0+3+11, 2);
		\coordinate (E) at (1+4+10, 0.5);
		\draw [very thick] (D)  to[out=205, in=90]   (B)  to[out=-90, in=180] (E);


		\coordinate (d) at (3+21-9, 1.5);
		\coordinate (e) at (5+21-9, 0.7);
		\coordinate[label=above:1] (f) at (4+21-9, 1.5); 
		\draw [very thick] (d) to[out=205, in=90] (2+12, 0.7) to[out=-90, in=-90] (e) to[out=90, in=-25] (f);

		
		\coordinate (x) at (19-2.5,0.5);
		\coordinate (y) at (20.5-3, 2.5);
		\coordinate[label=right:0] (z) at (21+0.5-3, 0);
		\draw [very thick] (x) to[in=-90, out=25] (y) to[in=150, out=-90] (z);

		\node[below] at (19.9+0.7-3, -0.5+0.8) {$+$};
		\node[right] at (20.3-0.7-3, 1.2+0.6+1) {$A_4$};

		\draw[very thick, ->,
		line join=round,
		decorate, decoration={
			zigzag,
			segment length=4,
			amplitude=.9, post=lineto,
			post length=2pt
		}] (20, 0.5) -- (21, 0.5);

		
		\coordinate (x) at (-0.5+21+2,2);
		\coordinate (y) at (-0.5+21+2, -2.5);
		\draw [very thick, in=105, out=-45] (x) to (y);
		
		\coordinate (a) at (-1+21+2, 0);
		\coordinate (b) at (2+21+2, 0.7);
		\coordinate[label=above:0] (c) at (1+21+2, 1.5);
		\draw [very thick] (a) to[out=0, in=-90] (b) to[out=90, in=-25] (c);
		\node[below] at (2+21+2, 0.2) {$-$};
		\coordinate (d) at (3+21+2, 1.5);
		\coordinate (e) at (5+21+2, 0.7);
		\coordinate[label=above:0] (f) at (4+21+2, 1.5); 
		\draw [very thick] (d) to[out=205, in=90] (b) to[out=-90, in=-90] (e) to[out=90, in=-25] (f);
		
		\coordinate (A) at (-1+12+14+2, 0.5);
		\coordinate (B) at (2+11+14+2, 1.2);
		\coordinate[label=above:0] (C) at (1++11+14+2, 2);
		\draw [very thick] (A)   to[out=0, in=-90]  (B)  to[out=90, in=-25] (C);
		\node[below] at (2+11+14+2, 0.7) {$+$};
		
		\coordinate[label=above:0] (D) at (0+3+11+14+2, 2);
		\coordinate (E) at (1+4+10+14+2, 0.5);
		\draw [very thick] (D)  to[out=205, in=90]   (B)  to[out=-90, in=180] (E);
		
		\coordinate (d) at (3+21+4.5+2, 1.5);
		\coordinate (e) at (5+21+4.5+2, 0.7);
		\coordinate[label=above:0] (f) at (4+21+4.5+2, 1.5); 
		\draw [very thick] (d) to[out=205, in=90] (2+21+5+2, 0.7) to[out=-90, in=-90] (e) to[out=90, in=-25] (f);
		\node[below] at (4+21+5.5+2, 0.2) {$-$};
		\coordinate[label=above:0] (l) at (6+21+4.5+2, 1.5);
		\coordinate (m) at (8+21+4.5+1.3, 0);
		\draw [very thick] (l) to[out=205, in=90] (5+21+4.5+2, 0.7) to[out=-90, in=180] (m);
		
		
		\coordinate (x) at (19-2.5+17.2,-0.5);
		\coordinate (y) at (20.5-3+18, 2.5);
		\coordinate[label=right:0] (z) at (21+0.5-3+18, 0);
		\draw [very thick] (x) to[in=-90, out=25] (y) to[in=150, out=-90] (z);

		\node[below] at (19.9+0.7-3+18, -0.5+0.8) {$+$};
		\node[right] at (20.3-0.7-3+18, 1.2+0.6+1) {$A_4$};
	\end{tikzpicture}
	\caption{A curve with an $A_1/A_4$-attached elliptic bridge of length $r=2$. First the $A_3$-singularity is replaced by an $\ov X(A_3)$-atom and the $A_4$-singularity by an $\ov X(A_4)$-atom, then the elliptic bridges are replaced by  $\ov X(^{\opp}A_3)$-atoms.}\label{F:A4ellipticbridge}
\end{figure}

We now compute the weight at $T$ as the sum of the contributions of the $(r-1)$ atoms $\ov X(A_3)$,  of the $r$ atoms $\ov X(^{\opp}A_3)$ and of the single atom $\ov X(A_4)$ in the central fiber:
$$\begin{aligned}
& \wei_T(K+\psi+\alpha(\delta-\psi)) 
=(r-1)w_{\alpha}(A_3)-rw_{\alpha}(A_3)+w_{\alpha}(A_4)=\\
& =-w_{\alpha}(A_3)+w_{\alpha}(A_4)= 
-(7-10\alpha)+(26-39\alpha)=19-29\alpha<0,    
\end{aligned}$$
since $\alpha>19/29$, which contradicts the $\alpha$-semistability of $(C,p_i)$.

Suppose, by contradiction, that $(C,p_i)$ contains an $A_4/A_4$-attached elliptic chain $E= E_1\cup\cdots\cup E_r$, with $r\geq 1$ bridges and
$r-1$ internal tacnodes, and two $A_4$-singularities at the two ends. The $C$ is itself equal to the entire chain, and hence it has odd arithmetic genus $2r+3\geq 5$ and $n=0$. We first apply Theorem \ref{T:deg-sing} to replace each of the  $r-1$ internal tacnodes by a $\ov X(A_3)$-atom and each of the two $A_4$ at the ends by a $\ov X(X_4)$-atom. The resulting curve consists of a chain of $A_1/A_1$-attached elliptic bridges joined by $\ov X(A_3)$-atoms and with two $\ov X(A_4)$-atoms at the end. 
We then replace each $A_1/A_1$-attached elliptic bridge with a $\ov X(^{\opp}A_3)$-atom as in the construction (B), see Figure \ref{F:A_4/A_4}.

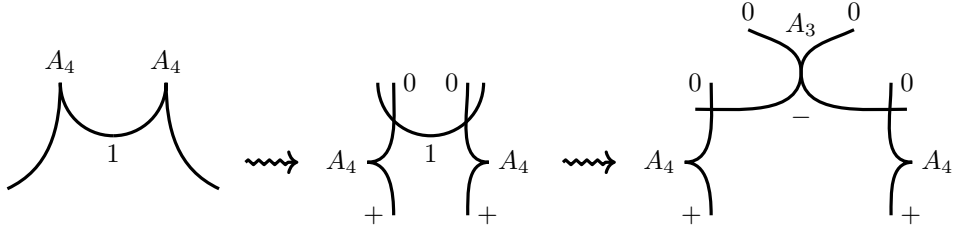
\begin{figure}[!h]
	\begin{center}
		\begin{tikzpicture}[scale=0.7]
			
			\coordinate (x) at (0,0);
			\coordinate[label=above: $A_4$] (y) at (1, 2);
			\coordinate[label=below:1] (z) at (2, 1);
			\draw [very thick] (x) to[in=-90, out=25] (y) to[in=180, out=-90] (z);

			\coordinate[label=above: $A_4$] (u) at (3,2);
			\coordinate (v) at (4, 0);
			\draw [very thick] (z) to[in=-90, out=0] (u) to[in=155, out=-90] (v);
			
			\draw[very thick, ->,
			line join=round,
			decorate, decoration={
				zigzag,
				segment length=4,
				amplitude=.9, post=lineto,
				post length=2pt
			}] (4.5, 0.5) -- (5.5, 0.5);

			
			\coordinate (a) at (0+7, 2);
			\coordinate[label=below:1] (b) at (1+7, 1);
			\coordinate (c) at (2+7,2);
			\draw [very thick] (a) to[in=180, out=-90] (b) to[in=-90, out=0] (c);
			
			\coordinate[label=right:0] (x) at (0+7+0.3, 2);
			\coordinate[label=left:$A_4$] (y) at (1+7-1.2, 0.5);
			\coordinate[label=left:$+$] (z) at (0+7+0.3,-0.5);
			\draw [very thick] (x) to[in=0, out=-90] (y) to[in=90, out=0] (z);
			
			\coordinate[label=left:0] (x) at (0+7+0.2+1.5, 2);
			\coordinate[label=right:$A_4$] (y) at (1+7+1.1, 0.5);
			\coordinate[label=right:$+$] (z) at (0+7+0.2+1.5,-0.5);
			\draw [very thick] (x) to[in=180, out=-90] (y) to[in=90, out=180] (z);

			\draw[very thick, ->,
			line join=round,
			decorate, decoration={
				zigzag,
				segment length=4,
				amplitude=.9, post=lineto,
				post length=2pt
			}] (4.5+6, 0.5) -- (5.5+6, 0.5);
			

			\coordinate[label=left:0] (x) at (0+7+0.3+6, 2);
			\coordinate[label=left:$A_4$] (y) at (1+7-1.2+6, 0.5);
			\coordinate[label=left:$+$] (z) at (0+7+0.3+6,-0.5);
			\draw [very thick] (x) to[in=0, out=-90] (y) to[in=90, out=0] (z);
			
			\coordinate (A) at (-1+12+2, 0.5+1);
			\coordinate (B) at (2+11+2, 1.2+1);
			\coordinate[label=above:0] (C) at (1+11+2, 2+1);
			\draw [very thick] (A)   to[out=0, in=-90]  (B)  to[out=90, in=-25] (C);
			\node[below] at (2+11+2, 0.7+1) {$-$};
			\node[above] at (2+11+2, 1.2+1+0.5) {$A_3$};
			\coordinate[label=above:0] (D) at (0+3+11+2, 2+1);
			\coordinate (E) at (1+4+10+2, 0.5+1);
			\draw [very thick] (D)  to[out=205, in=90]   (B)  to[out=-90, in=180] (E);

			\coordinate[label=right:0] (x) at (0+7+0.2+1.5+8, 2);
			\coordinate[label=right:$A_4$] (y) at (1+7+1.1+8, 0.5);
			\coordinate[label=right:$+$] (z) at (0+7+0.2+1.5+8,-0.5);
			\draw [very thick] (x) to[in=180, out=-90] (y) to[in=90, out=180] (z);

		\end{tikzpicture}
	\end{center}
	\caption{An $A_4/A_4$ attached elliptic bridge of length $r=1$. The curve has genus 5. First the $A_4$-singularities are replaced by $\ov X(A_4)$-atoms and then the elliptic bridge is replaced by an $\ov X(^{\opp}A_3)$-atom.}\label{F:A_4/A_4}
\end{figure}

We now compute the weight at $T$ as the sum of the contributions of the $(r-1)$ atoms $\ov X(A_3)$,  of the $r$ atoms $\ov X(^{\opp}A_3)$ and of the two atoms $\ov X(A_4)$ in the central fiber:
$$\begin{aligned}
& \wei_T(K+\psi+\alpha(\delta-\psi)) 
=(r-1)w_{\alpha}(A_3)-rw_{\alpha}(A_3)+2w_{\alpha}(A_4)=\\
& =-w_{\alpha}(A_3)+2w_{\alpha}(A_4)= 
-(7-10\alpha)+2(26-39\alpha)=45-68\alpha<0,    
\end{aligned}$$
since  $\alpha>45/68$., which contradicts the $\alpha$-semistability of $(C,p_i)$.

\item Assume $\alpha \in (2/3 -\varepsilon, 2/3)$. Suppose, by contradiction, that $(C,p_i)$ contains an $A_1$-attached Weierstrass tail $E=E_1 \cup \cdots \cup E_r \cup D$, where $E_1, \ldots, E_r$ are elliptic curves and $D$ is a genus 2 curve. The attaching points are $A_3$-singularities and $E_r \cap D$ is a Weierstrass point of $D$.
We first apply Theorem \ref{T:deg-sing} to replace each of the  $r$ internal tacnodes by a $\ov X(A_3)$-atom. The resulting curve consist of a chain of $r$ $A_1/A_1$-attached elliptic bridges and one $A_1$-attached Weierstrass tail, joined by $r$ $\ov X(A_3)$-atoms. We then replace each $A_1/A_1$-attached elliptic bridge with a $\ov X(^{\opp}A_3)$-atom as in the construction (B) and the $A_1$-attached Weierstrass tail with a $\ov X(^{\opp}A_4)$-atom as in the construction (C), see Figure \ref{F:weierstrass}.

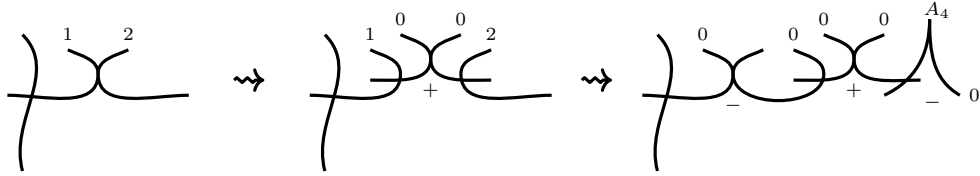
\begin{figure}[!h]
	\centering
	\begin{tikzpicture}[scale=0.4, every node/.style={font=\scriptsize}]
		
		\coordinate (x) at (-0.5,2);
		\coordinate (y) at (-0.5, -2.5);
		\draw [very thick, in=105, out=-45] (x) to (y);
		
		\coordinate (a) at (-1, 0);
		\coordinate (b) at (2, 0.7);
		\coordinate[label=above:1] (c) at (1, 1.5);
		\draw [very thick] (a)   to[out=0, in=-90]  (b)  to[out=90, in=-25] (c);

		\coordinate[label=above:2] (d) at (0+3, 1.5);
		\coordinate (e) at (1+4, 0);
		\draw [very thick] (d)  to[out=205, in=90]   (b)  to[out=-90, in=180] (e);

		\draw[very thick, ->,
		line join=round,
		decorate, decoration={
			zigzag,
			segment length=4,
			amplitude=.9, post=lineto,
			post length=2pt
		}] (6.5, 0.5) -- (7.5, 0.5);

		\coordinate (x) at (-0.5+10,2);
		\coordinate (y) at (-0.5+10, -2.5);
		\draw [very thick, in=105, out=-45] (x) to (y);
		
		\coordinate (a) at (-1+10, 0);
		\coordinate (b) at (2+10, 0.7);
		\coordinate[label=above:1] (c) at (1+10, 1.5);
		\draw [very thick] (a)   to[out=0, in=-90]  (b)  to[out=90, in=-25] (c);
		
		\coordinate (A) at (-1+12, 0.5);
		\coordinate (B) at (2+11, 1.2);
		\coordinate[label=above:0] (C) at (1+11, 2);
		\draw [very thick] (A)   to[out=0, in=-90]  (B)  to[out=90, in=-25] (C);
		\node[below] at (2+11, 0.7) {$+$};
		
		\coordinate[label=above:0] (D) at (0+3+11, 2);
		\coordinate (E) at (1+4+10, 0.5);
		\draw [very thick] (D)  to[out=205, in=90]   (B)  to[out=-90, in=180] (E);

		\coordinate[label=above:2] (d) at (0+3+12, 1.5);
		\coordinate (e) at (1+4+12, 0);
		\draw [very thick] (d)  to[out=205, in=90]   (2+12, 0.7)  to[out=-90, in=180] (e);

		\draw[very thick, ->,
		line join=round,
		decorate, decoration={
			zigzag,
			segment length=4,
			amplitude=.9, post=lineto,
			post length=2pt
		}] (18, 0.5) -- (19, 0.5);

		
		\coordinate (x) at (-0.5+21,2);
		\coordinate (y) at (-0.5+21, -2.5);
		\draw [very thick, in=105, out=-45] (x) to (y);
		
		\coordinate (a) at (-1+21, 0);
		\coordinate (b) at (2+21, 0.7);
		\coordinate[label=above:0] (c) at (1+21, 1.5);
		\draw [very thick] (a) to[out=0, in=-90] (b) to[out=90, in=-25] (c);
		\node[below] at (2+21, 0.2) {$-$};
		\coordinate (d) at (3+21, 1.5);
		\coordinate (e) at (5+21, 0.7);
		\coordinate[label=above:0] (f) at (4+21, 1.5); 
		\draw [very thick] (d) to[out=205, in=90] (b) to[out=-90, in=-90] (e) to[out=90, in=-25] (f);
		
		\coordinate (A) at (-1+12+14, 0.5);
		\coordinate (B) at (2+11+14, 1.2);
		\coordinate[label=above:0] (C) at (1++11+14, 2);
		\draw [very thick] (A)   to[out=0, in=-90]  (B)  to[out=90, in=-25] (C);
		\node[below] at (2+11+14, 0.7) {$+$};
		
		\coordinate[label=above:0] (D) at (0+3+11+14, 2);
		\coordinate (E) at (1+4+10+14.2, 0.5);
		\draw [very thick] (D)  to[out=205, in=90]   (B)  to[out=-90, in=180] (E);
		

\coordinate (x) at (19+9,0);
		\coordinate (y) at (20+9.5, 2.5);
		\coordinate[label=right:0] (z) at (21+9.5, 0);
		\draw [very thick] (x) to[in=-90, out=25] (y) to[in=150, out=-90] (z);

		\node[below] at (19.9+9.7, -0.5+0.8) {$-$};
		\node[right] at (20.3+8.7, 1.2+0.6+1) {$A_4$};
		
	\end{tikzpicture}
	\caption{A curve with an $A_1$-attached Weierstrass tail $E=E_1 \cup D$.  First the $A_3$-singularity is replaced by a $\ov X(A_3)$-atom, then the elliptic bridge is replaced by  a $\ov X(^{\opp}A_3)$-atom and the genus 2 curve by $\ov X(^{\opp}A_4)$-atom.}\label{F:weierstrass}
\end{figure}
We now compute the weight at $T$ as the sum of the contributions of the $r$ atoms $\ov X(A_3)$,  of the $r$ atoms $\ov X(^{\opp}A_3)$ and of  atoms $\ov X(^{\opp}A_4)$ in the central fiber:
$$\begin{aligned}
& \wei_T(K+\psi+\alpha(\delta-\psi)) 
=rw_{\alpha}(A_3)-rw_{\alpha}(A_3)-w_{\alpha}(A_4)=\\
& =-w_{\alpha}(A_4)= 39\alpha-26<0,    
\end{aligned}$$
since  $\alpha<2/3$, contradicting the $\alpha$-semistability of $(C,p_i)$.
\end{enumerate}

\end{proof}

\begin{remark}\label{R:threshold}
The proof of Theorem \ref{thm:U(alpha)} shows that  we have
$$
\mc U_{g,n}(\alpha) \subset \ov{\mc M}_{g,n}(2/3 -\varepsilon) \text{ for } \alpha \in 
\begin{cases}
   (17/28,2/3) & \text{ if } (g,n)=(3,0),\\
(45/68, 2/3) & \text{ if } g=2r+3\geq 5 \text{ and } n=0,\\
(19/29, 2/3) & \text{ otherwise.}
\end{cases}
$$
In particular, it suggests that for odd genus $g\geq 5$ and $n=0$, there is a threshold values at $\alpha=45/68$. This threshold seems to be new, in the sense that it does not appear in the predictions of \cite[Table 3]{AFS16}, \cite[\S 4.4]{FS}, \cite[Table 1]{AlpHye}, \cite[Prop. 1]{Hye17}, \cite[Problem 5.7]{AHL}.
\end{remark}

\bibliographystyle{amsalpha}
\bibliography{Library}

\end{document}